\documentclass[reqno,oneside,12pt]{amsart}
\usepackage{
  amsmath,
  amssymb,
  amsthm,
  geometry,
  paralist,
  thmtools,
  tikz,
  tikz-cd,
  verbatim,
  relsize, 
  bbm, 
  dsfont,
  mathtools,
  stmaryrd
}

\usepackage[T1]{fontenc}

\usepackage{caption}
\usepackage{subcaption}
\usetikzlibrary{angles}

\renewcommand\labelenumi{(\roman{enumi})}
\renewcommand\theenumi\labelenumi

\usepackage{tcolorbox}

\RequirePackage{ifpdf}
\ifpdf
  \IfFileExists{pdfsync.sty}{\RequirePackage{pdfsync}}{}
  \RequirePackage[pdftex,
   colorlinks = true,
   urlcolor = myblue, 
   citecolor = mygreen, 
   linkcolor = myred, 
   bookmarksopen=true,
   unicode, psdextra]{hyperref}
\else
  \RequirePackage[hypertex, unicode, psdextra]{hyperref}
\fi
\hypersetup{
    colorlinks=true,
    linkcolor=blue,
    filecolor=magenta,      
    urlcolor=cyan,
    citecolor = blue
} 

\RequirePackage{ae, aecompl, aeguill} 

\usepackage{quiver}

\usepackage{cleveref} 

\usepackage{adjustbox}

\tikzcdset{
	arrows = crossing over
}

\newtheorem{thmx}{Theorem}

\newtheorem{corx}[thmx]{Corollary}

\usepackage[maxbibnames=99,
			maxalphanames=4, 
			style=alphabetic]{biblatex}
\usepackage{xurl}

\usepackage{dynkin-diagrams}

\usepackage[boxsize=.5em]{ytableau}

\usetikzlibrary{arrows.meta} 

\usepackage[textsize=scriptsize]{todonotes}
\presetkeys{todonotes}{inline}{}

\makeatletter
\providecommand\@dotsep{5}
\makeatother

\newcommand{\bbC}{\mathbb C}
\newcommand{\bbH}{\mathbb H}

\newcommand{\bbQ}{\mathbb Q}
\newcommand{\bbR}{\mathbb R}

\newcommand{\bbZ}{\mathbb Z}

\newcommand{\cA}{\mathcal A}

\newcommand{\cC}{\mathcal C}
\newcommand{\cD}{\mathcal D}
\newcommand{\cF}{\mathcal F}

\newcommand{\cH}{\mathcal H}
\newcommand{\cI}{\mathcal I}

\newcommand{\cM}{\mathcal M}
\newcommand{\cO}{\mathcal O}
\newcommand{\cP}{\mathcal P}
\newcommand{\cQ}{\mathcal Q}
\newcommand{\cU}{\mathcal U}

\newcommand{\cZ}{\mathcal Z}

\newcommand{\catname}[1]{{\mathsf{#1}}}

\newcommand{\TLJ}{\catname{TLJ}}
\renewcommand{\mod}{\catname{mod}}

\renewcommand{\vec}{\catname{vec}}
\newcommand{\rep}{\catname{rep}}
\newcommand{\lmod}{\text{-}\catname{mod}}
\newcommand{\lMod}{\text{-}\catname{Mod}}
\newcommand{\lMOD}{\text{-}\catname{MOD}}
\DeclareMathOperator{\cEnd}{\catname{End}}
\newcommand{\Fib}{\catname{Fib}}
\newcommand{\Fuk}{\catname{Fuk}}
\DeclareMathOperator{\Cat}{\text{Cat}}

\newcommand{\wt}{\widetilde}

\newcommand{\leftdual}[1]{{}^\vee{#1}}
\newcommand{\rightdual}[1]{{#1}^\vee}

\DeclareMathOperator{\id}{id}
\DeclareMathOperator{\Z}{\mathbb Z}
\DeclareMathOperator{\R}{\mathbb R}

\DeclareMathOperator{\C}{\mathbb C}
\newcommand{\1}{\mathds{1}}

\DeclareMathOperator{\ra}{\rightarrow}
\newcommand{\xra}{\xrightarrow}

\DeclareMathOperator{\Aut}{Aut}
\DeclareMathOperator{\Cone}{Cone}
\DeclareMathOperator{\Hom}{Hom}

\DeclareMathOperator{\rank}{rank}
\DeclareMathOperator{\Stab}{Stab}

\DeclareMathOperator{\Irr}{Irr}

\DeclareMathOperator{\FPdim}{FPdim}
\DeclareMathOperator{\Dom}{Dom}
\DeclareMathOperator{\mult}{mult}
\DeclareMathOperator{\SU}{SU}

\DeclareMathOperator{\Ker}{Ker}
\newcommand{\QCoh}{\mathrm{QCoh}}
\newcommand{\Coh}{\mathrm{Coh}}
\DeclareMathOperator*{\alb}{alb}
\DeclareMathOperator*{\Alb}{Alb}
\newcommand{\Pic}{\mathrm{Pic}}
\newcommand{\Knum}{K_{\mathrm{num}}} 

\newcommand{\Stabgeo}{\Stab^{\mathrm{Geo}}}
\newcommand{\characteristic}{\text{\normalfont char}}
\newcommand{\ch}{\text{\normalfont ch}} 
\newcommand{\lambdanat}{\lambda_{\text{\normalfont nat}}} 

\declaretheorem[numberwithin=section]{theorem}
\declaretheorem[sibling=theorem]{lemma}
\declaretheorem[sibling=theorem]{corollary}

\declaretheorem[sibling=theorem]{proposition}
\declaretheorem[sibling=theorem]{question}
\declaretheorem[sibling=theorem, style=remark]{remark}
\declaretheorem[sibling=theorem, style=definition]{definition}
\declaretheorem[sibling=theorem, style=definition]{example}
\declaretheorem[sibling=theorem, style=definition]{assumption}

\newtheorem*{theorem*}{Theorem}

\crefname{lemma}{Lemma}{Lemma}
  \crefname{corollary}{Corollary}{Corollary}
  \crefname{theorem}{Theorem}{Theorem}
  \crefname{definition}{Definition}{Definition}
   \crefname{proposition}{Proposition}{Proposition}
\crefname{question}{Question}{Question}
 \crefname{section}{Section}{Section} 
   \crefname{construction}{Construction}{Construction}
   \crefname{generalization}{Generalization}{Generalization}
  \crefname{construction}{Construction}{Construction}
  \crefname{notation}{Notation}{Notation}
   \crefname{example}{Example}{Example}
  \crefname{remark}{Remark}{Remark}
  \crefname{fact}{Fact}{Fact}
  \crefname{conjecture}{Conjecture}{Conjecture}
  \crefname{motivation}{Motivation}{Motivation}  
  \crefname{figure}{Figure}{Figure}  
  \crefname{assumption}{Assumption}{Assumption}

\newcommand{\nPi}[1][n]{{ \overset{\scriptscriptstyle #1}{\Pi}}}

\renewcommand{\comment}[1]{}

\newcommand*\cocolon{%
        \nobreak
        \mskip6mu plus1mu
        \mathpunct{}%
        \nonscript
        \mkern-\thinmuskip
        {:}%
        \mskip2mu
        \relax
}

\let\Re\relax
\DeclareMathOperator{\Re}{\mathrm{Re}} 
\let\Im\relax
\DeclareMathOperator{\Im}{\mathrm{Im}} 

\begin{document}

\title[]{Fusion-equivariant stability conditions and Morita duality}

\author[]{Hannah Dell}
\address{School of Mathematics and Maxwell Institute, University of Edinburgh, James Clerk Max\-well Building, Peter Guthrie Tait Road, Edinburgh, EH9 3FD (UK)}
\email{h.dell@sms.ed.ac.uk}

\author[]{Edmund Heng}
\address{Institut des Hautes Etudes Scientifiques (IHES). Le Bois-Marie, 35, route de Chartres, 91440 Bures-sur-Yvette (France)}
\email{heng@ihes.fr}

\author[]{Anthony M. Licata}
\address{Mathematical Sciences Institute, Australian National University (ANU), Canberra (Australia)}
\email{anthony.licata@anu.edu.au}

\keywords{Stability conditions; fusion categories; finite group actions, Morita duality}

\subjclass{}

\begin{abstract}
	Given a triangulated category $\cD$ with an action of a fusion category $\cC$, we study the moduli space $\Stab_{\cC}(\cD)$ of fusion-equivariant Bridgeland stability conditions on $\cD$.  The main theorem is that the fusion-equivariant stability conditions form a closed, complex submanifold of the moduli space of stability conditions on $\cD$.  As an application of this framework to finite group actions on categories, we generalise a result of Macr\`{i}--Mehrotra--Stellari by establishing a biholomorphism between the space of $G$-invariant stability conditions on $\cD$ and the space of $\rep(G)$-equivariant stability conditions on the equivariant category $\cD^G$.  We also describe applications to the study of stability conditions associated to McKay quivers and to geometric stability conditions on free quotients of smooth projective varieties.
\end{abstract}

\maketitle

\tableofcontents

\section{Introduction}

\subsection{Fusion-equivariant stability conditions}

The moduli space $\Stab(\cD)$ of Bridgeland stability conditions on a triangulated category $\cD$, introduced in \cite{bridgeland_2007}, associates a complex manifold to a triangulated category.  
When $\cD$ has additional symmetry, one hopes to see that reflected within the stability manifold.  For example, when $\cD$ is equipped with a group action, one can consider the invariant stability conditions, which form a closed submanifold of $\Stab(\cD)$ \cite[Theorem 1.1]{MMS_09}. 

In this paper we consider a more general situation that appears in both algebraic geometry and representation theory, namely, when the triangulated category $\cD$ is acted on by a fusion category $\cC$ (see \cref{defn: module category}). 
The situation where $\cD$ is equipped with an action of a group $G$ is recovered as the special case  $\cC = \vec_G$, but there are a number of interesting examples that don't come directly from group actions.  When a fusion category $\cC$ acts on $\cD$, its fusion ring $K_0(\cC)$ then acts on the Grothendieck group $K_0(\cD)$ and on the complex numbers $\bbC$; the latter action is given by the Frobenius--Perron homomorphism.  Thus, inside the space of central charges sits a subspace of fusion-equivariant central charges $Z: K_0(\cD) \longrightarrow \bbC$ that intertwine the action of the fusion ring.

The second author, in his thesis \cite{Heng_PhDthesis}, considered a lift of this consideration from central charges to stability conditions by defining a subset $\Stab_{\cC}(\cD)$ of $\cC$-equivariant stability conditions on $\cD$ (see \cref{defn: C equivariant stab}).  
In this paper we consider the basic topological properties of the subset $\Stab_{\cC}(\cD)$.  We prove the following.
\begin{thmx}(=\cref{cor: C-equivariant complex submanifold})
The space $\Stab_{\cC}(\cD)$ of $\cC$-equivariant stability conditions is a closed, complex submanifold of $\Stab(\cD)$.
\end{thmx}
\noindent In the special case where $\cC = \vec_G$ (i.e. $\cD$ has an action of $G$), we recover (part of) \cite[Theorem 1.1]{MMS_09}. We also adapt some of their techniques for our proof.
We note here that the submanifold property of $\Stab_{\cC}(\cD)$ was also obtained independently by Qiu and Zhang in \cite{QZ_fusion-stable}, where they considered fusion-equivariant stability conditions in relation to cluster theory.  

\subsection{A duality of equivariant stability conditions: \texorpdfstring{$\vec_G$ v.s.\ $\rep(G)$}{vecG v.s.\ rep(G)} }

In the case where $\cD$ is acted upon by a finite \emph{abelian} group $G$, it was known that there is an isomorphism (in fact, a biholomorphism) between $G$-invariant stability conditions on $\cD$ and $\widehat{G}$-invariant stability conditions on $\cD^G$, where $\widehat{G} \coloneqq \Hom(G,\Bbbk^\ast)$ is the Pontryagin dual of $G$ and $\cD^G$ is the category of $G$-equivariant objects; see  \cite[Lemma 4.11]{perryModuliSpacesStable2023} and \cite[Theorem 2.26]{dellStabilityConditionsFree2023}.  One immediate obstacle to extending this observation from abelian to nonabelian groups is that when $G$ is nonabelian, one needs to find a suitable replacement for the Pontryagin dual $\widehat{G}$.  This naturally brings in fusion categories: instead of $G$ and $\widehat{G}$, one should consider the (Morita) dual fusion categories $\vec(G)$ and $\rep(G)$. 
We prove the following duality theorem.
\begin{thmx}(=\cref{thm:moritastability})
	Let $\cD$ be a triangulated module category over $\vec_G$ subject to mild hypotheses.  The induction functor $\cI\colon \cD \ra \cD^G$ and the forgetful functor $\cF\colon \cD^G \ra \cD$ induce biholomorphisms
	\[
		 \Stab_{\rep(G)}(\cD^G) {\rightleftarrows} \Stab_{\vec(G)}(\cD),
	\]
	which are mutually inverse up to rescaling the central charge by $|G|$.
\end{thmx}

We note here that the fact that $\cF$ induces a closed embedding of $\Stab_{\vec_G}(\cD)$ into $\Stab(\cD^G)$ was already known \cite[Theorem 1.1]{MMS_09}.
Our result identifies the image of this embedding as being exactly the $\rep(G)$-equivariant stability conditions.  Another related previous result is \cite[Lemma 2.2.3]{Pol_constant-t}, which considers the case of an action of a finite group on a scheme $X$. Polishchuk describes a bijection between $\Stab_{\vec_G}(X)$ and stability conditions on $D^b\Coh^G(X)$, the derived category of $G$-equivariant sheaves, that are invariant under tensoring with the regular representation.  From our work it now follows that this bijection is a biholomorphism of closed complex submanifolds.

\subsection{On McKay quivers and free quotients}
There are a few reasons why the submanifold $\Stab_\cC(\cD)$ might be worthy of further consideration.  One simple reason is that the submanifold $\Stab_\cC(\cD)$ might be easier to understand than the entire manifold $\Stab(\cD)$, while still capturing something important about $\cD$ or its autoequivalence group.  

In \cref{sec:app}, we present two classes of examples in relation to fusion categories coming from finite groups $G$; we describe them both briefly here.  
The first class of examples comes from separated quivers $\overline{Q}_{G,V}$ of McKay quivers $Q_{G,V}$, which were studied e.g.\ in \cite{AusRei_McKay} (see \cref{defn:McKayquiver} and \cref{fig:McKayquiver}).
We fully identify the closed submanifold of $\rep(G)$-equivariant stability conditions of $\overline{Q}_{G,V}$ via the stability manifold of the $\ell$-Kronecker quiver:

\begin{corx}(=\cref{cor:McKaystability})
	Let $\ell \coloneqq \dim_\Bbbk(V)$. We have
	\[
	\Stab_{\rep(G)}(D^b\overline{Q}_{G,V}) \cong \Stab(D^bK_\ell) \cong 
		\begin{cases}
		\C^2, &\text{if } \ell \leq 2;\\
		\C \times \bbH, &\text{if } \ell \geq 3,
		\end{cases}
	\]
\end{corx}
\noindent The final isomorphism is due to the works \cite{Qiu_PhDthesis,bridgeland_qiu_sutherland_2020,Okada_P1,DK_KroneckerStab}.
Note that the generalised McKay correspondence in \cite{AusRei_McKay} shows that separated McKay quivers $\overline{Q}_{G,V}$ are wild quivers once $\ell \geq 3$, and at present the homotopy type of the entire stability manifold of such quivers is not generally known. 

The second class of examples is of a more geometric flavour. Let $X$ be a smooth projective variety, and let $G$ be a finite group acting freely on $X$. This induces an action of $G$ on $D^b\Coh(X)$, and the $G$-equivariant category is equivalent to the derived category of the smooth quotient $X/G$, $D^b\Coh^G(X)\cong D^b\Coh(X/G)$. 
When $G$ is abelian, the relationship between $\Stab(X)$ and $\Stab(X/G)$ was used by the first author in \cite{dellStabilityConditionsFree2023} to give new examples of situations where a stability manifold has a connected component consisting entirely of geometric stability conditions, which are those whose skyscraper sheaves of points $\cO_x$ are all stable. Here we generalise this to the non-abelian case. We first prove that the isomorphisms of Theorem \ref{thm:moritastability} preserve geometric stability conditions:
\begin{corx}(=\cref{cor:geometriccorrespondence})
	There are natural biholomorphisms
	\[
		\Stab_{\rep(G)}(X/G) \overset{\cong}{\rightleftarrows} \Stab_{\vec_G}(X),
	\]
	which are mutually inverse up to rescaling the central charge by $|G|$. Moreover, these isomorphisms preserve the locus of geometric stability conditions.
\end{corx}
Recall that every algebraic variety $X$ has a morphism $\alb_X$, the \emph{Albanese morphism}, to an abelian variety $\Alb(X):=\Pic^0(\Pic^0(X))$, called the \emph{Albanese variety}. Every morphism $f\colon X\rightarrow A$ to another abelian variety $A$ factors via $\alb_X$. In \cite[Theorem 1.1]{fuStabilityManifoldsVarieties2022}, the authors showed that if $X$ has finite Albanese morphism, then all stability conditions on $D^b\Coh(X)$ are geometric. In \cref{thm: Albanese connected component}, we use the above isomorphisms to obtain a union of connected components of geometric stability conditions on any free quotient of a variety with finite Albanese morphism. In particular, as in the abelian case \cite[Examples 3.12 and 3.14]{dellStabilityConditionsFree2023}, this produces new examples of varieties with non-finite Albanese morphism whose stability manifolds have a connected component of only geometric stability conditions (see Examples \ref{example: CY 3fold of abelian type}, \ref{example: GHVs}, and \ref{example: non abelian Beauville-type}). Since the stability manifold of any variety is expected to be connected, this provides further evidence that the following question has a negative answer.

\begin{question}[{\cite[Question 4.11]{fuStabilityManifoldsVarieties2022}}]\label{Question FLZ}
	Let $X$ be a smooth projective variety with $\alb_X$ non-finite. Are there always non-geometric stability conditions on $D^b\Coh(X)$?
\end{question}

\subsection{Fusion-equivariant stability conditions in Coxeter theory}
Another class of examples comes from Coxeter theory, the rank two examples of which were introduced in \cite{Heng_PhDthesis}.  While these examples are not studied directly in the current paper, they are important motivation for our work, so we describe them briefly here.

If $(W,S)$ is a Coxeter system, then there is an associated 2-Calabi--Yau category $\cD_W$, which always has a non-empty stability manifold $\Stab(\cD_W)$.  When $W$ is simply-laced (e.g. finite type ADE), this 2-Calabi--Yau category is well-studied \cite{bridgeland2009kleinian, brav_thomas_2010, ikeda2014stability}, and can be defined using representations of the preprojective algebra of the associated quiver (or, quadratic-dually, via representations of the associated zigzag algebra).  For simply-laced $W$, a connected component of the stability manifold is a covering space of the associated complexified hyperplane complement:
\[
\pi: \Stab(\cD_W) \longrightarrow \mathfrak{h}_{reg}/W,
\]
(Here $\mathfrak{h}_{reg}$ denotes the complement of the root hyperplanes in the complexified Tits cone.) This is important in part because if $\Stab(\cD_W)$ is contractible -- as is conjectured and proven for finite ADE types \cite{QW_18, AW_22} (independent of Deligne's proof of the $K(\pi,1)$ conjecture for spherical type Artin--Tits groups \cite{Deligne_72}) -- then it follows that $\mathfrak{h}_{reg}/W$ is a $K(\pi,1)$ for the associated Artin-Tits braid group.  Even showing that $\Stab(\cD_W)$ is simply-connected would solve the word problem in simply-laced Artin--Tits groups (see  e.g.\ \cite[Proposition 7.11]{hengstabilitycoxeter}), a long-standing open problem.  Thus, in simply-laced type, stability conditions on the 2-Calabi--Yau category arise naturally in the study of the associated braid group.

When $W$ is not simply-laced, however, it is less clear how to produce a moduli space of stability conditions which covers $\mathfrak{h}_{reg}/W$, in part because it might not be immediately obvious what the definition of the 2-Calabi--Yau category of a non-simply-laced Coxeter group should be.  In \cite{hengstabilitycoxeter}, we define a 2-Calabi--Yau category $\cD_W$ for an \emph{arbitrary} Coxeter system by first building zigzag and preprojective algebras in an appropriate fusion category $\cC_W$.
This generalises the rank two case studied earlier in \cite[Chapter 2]{Heng_PhDthesis}.
 As a direct consequence of the construction, the fusion category $\cC_W$ acts on the triangulated category $\cD_W$.  Thus one might hope that the moduli space of stability conditions on $\cD_W$ might cover $\mathfrak{h}_{reg}/W$ for all $W$.  However, in non-simply-laced type, the dimension of the entire stability manifold $\Stab(\cD_W)$ is too large to admit a covering map to $\mathfrak{h}_{reg}/W$.  The submanifold of 
$\cC_W$-equivariant stability conditions, however, does admit such a map.  We prove in \cite{hengstabilitycoxeter} that for any Coxeter group $W$, 
$\Stab_{\cC_W}(\cD_W)$ is a covering of $\mathfrak{h}_{reg}/W$.
Moreover, we conjecture that this covering is universal, and is contractible.  
For a similar approach to the non-simply-laced finite type cases, see \cite{QZ_fusion-stable}.
All in all, to study $K(\pi,1)$s for arbitrary Coxeter groups, it is therefore useful to work fusion-equivariantly.

In fact, we expect the entire picture described above for Coxeter systems to eventually be presented as a small part of the theory of $J$-equivariant stability conditions on cell quotients of the Hecke category of $W$, where $J$ is the categorified $J$-ring of Lusztig.

\subsection*{Notation and Conventions} 
$\Bbbk$ will be an algebraically closed field, $\cA$ will be an abelian category, and $\cD$ will be a triangulated category. $D^b\cA $ will denote the bounded derived category of $\cA$. $\Stab(\cD)$ will denote the space of all stability conditions on $\cD$ that are locally finite, except in \cref{subsection: free quotients} and \cref{section: support property} where we assume the support property.

\subsection*{Acknowledgements}
The authors thank the referee for suggestions and improvements, and especially for pointing out that our identification of closed submanifolds of stability conditions is in fact biholomorphic.
The authors would like to thank Arend Bayer, Asilata Bapat, Anand Deopurkar, Ailsa Keating, Abel Lacabanne, Clement Delcamp, Franco Rota, and Michael Wemyss for interesting and helpful discussions.  
We would also like to thank Yu Qiu for informing us about his related work with XiaoTing Zhang.
Part of this project began from a discussion of E. H. and H. D. during the conference ``Quivers, Cluster, Moduli and Stability'' at Oxford, for which the authors would like to thank the organisers for making this possible.
H.D. is supported by the ERC Consolidator Grant WallCrossAG, No. 819864.

\section{Preliminaries on fusion categories and actions}
\subsection{Fusion categories}
We shall quickly review the notion of a fusion category and some of its (desired) properties. 
We refer the reader to \cite{etingof_nikshych_ostrik_2005, EGNO15} for more details on this subject.

\begin{definition}\label{defn: tensor fusion category}
A (strict) \emph{fusion category} $\cC$ is a $\Bbbk$-linear, hom-finite, abelian category with a (strict) monoidal structure $(- \otimes -, \1)$ satisfying the following properties:
\begin{enumerate}
\item the monoidal unit $\1$ is simple;
\item $\cC$ is semisimple;
\item $\cC$ has finitely many isomorphism classes of simple objects; and
\item every object $C \in \cC$ has a left dual $\leftdual{C}$ and a right dual $\rightdual{C}$ with respect to the tensor product $\otimes$ (rigidity).
\end{enumerate}
A set of representatives of isomorphism classes of simple objects in $\cC$ will always be denoted by $\Irr(\cC)$.
\end{definition}

\begin{example}\label{eg:fusioncat}
\begin{enumerate}
\item The simplest example of a fusion category is the category of finite dimensional $\Bbbk$-vector spaces, $\vec$.
More generally, the category of $G$-graded vector spaces, $\vec_G$, for a finite group $G$ is also fusion, with simples given by $g \in G$ representing the one-dimensional vector space with grading concentrated in $g$. \label{eg:vecG}
\item Let $G$ be a finite group. If $\characteristic(\Bbbk) \nmid |G|$, then the category of representations of $G$, $\rep(G)$, is a fusion category. If $\characteristic(\Bbbk) \mid |G|$ and $|G|>1$, then $\rep(G)$ is not a fusion category since it is not semisimple \cite[Proposition 3.2]{etingofIntroductionRepresentationTheory2011} \label{eg:repG}
\item The Fibonacci (golden ratio) fusion category $\Fib$ is the (unique up to equivalence) fusion category with two simple objects $\1, \Pi$ and fusion rules 
\begin{equation} \label{eqn:goldenratiorel}
\Pi \otimes \Pi \cong \1 \oplus \Pi.
\end{equation}
This category can be explicitly realised as a fusion subcategory of the fusion category associated to $U_q(\mathfrak{s}\mathfrak{l}_2)$ at $q = e^{i\pi/5}$. \label{eg:Fib}
\end{enumerate} 
\end{example}

\begin{definition}
Let $\cA$ be an abelian category.
Its (exact) \emph{Grothendieck group} $K_0(\cA)$ is the abelian group freely generated by isomorphism classes $[A]$ of objects $A \in \cA$ modulo the relation
\[
[B] = [A] + [C] \iff 0 \ra A \ra B \ra C \ra 0 \text{ is exact}.
\]
Similarly, for $\cD$ a triangulated category, its (exact) \emph{Grothendieck group} $K_0(\cD)$ is the abelian group freely generated by isomorphism classes $[X]$ of objects $X \in \cD$ modulo the relation
\[
[Y] = [X] + [Z] \iff X \ra Y \ra Z \ra X[1] \text{ is a distinguished triangle}.
\]
When $\cA = \cC$ is moreover a fusion category, $K_0(\cC)$ can be equipped with a unital $\mathbb{N}$-ring structure.
Namely, it has basis given by the (finite) isomorphism classes of simples $\{ [S_j] \}_{j=0}^n$ and multiplication $[S_k]\cdot[S_i] \coloneqq [S_k \otimes S_i]$ defined by the tensor product that satisfies
\begin{equation} \label{eqn:simpledecomposition}
[S_k] \cdot [S_i] = \sum_{i=0}^n {}_k r_{i,j} [S_j], \quad {}_k r_{i,j} \in \mathbb{N},
\end{equation}
where $[S_0] = [\1]$ is the unit.
Together with this ring structure, we call $K_0(\cC)$ the \emph{fusion ring} of $\cC$.
\end{definition}

All non-zero objects of a fusion category have a well-defined notion of (not necessarily integral) dimension defined as follows.
\begin{definition}\label{defn:FPdim}
Let $\cC$ be a fusion category with $n+1$ simple objects $\{ S_k \}_{k=0}^n$.
For each simple $S_k$, we define $\FPdim([S_k]) \in \R$ to be the Frobenius--Perron eigenvalue (the unique largest real eigenvalue) of the non-negative matrix $({}_k r_{i,j})_{i,j \in \{0,1,...,n\}}$.
This extends to a ring homomorphism \cite[Proposition 3.3.6(1)]{EGNO15} that we denote by 
\[
\FPdim\colon K_0(\cC) \ra \R.
\]
The \emph{Frobenius--Perron dimension}, $\FPdim(X)$, of an object $X \in \cC$ is defined to be $\FPdim([X])$.
\end{definition}

The following proposition characterises the Frobenius--Perron dimension:
\begin{proposition}[\protect{\cite[Proposition 3.3.6(3)]{EGNO15}}] \label{prop:uniquedim}
Let $\cC$ be a fusion category.
There is at most one ring homomorphism $\varphi\colon K_0(\cC) \ra \R$ satisfying the following properties:
\begin{enumerate}
\item $\varphi([X]) \geq 0$ for all $X \in \cC$; and
\item $\varphi([X]) = 0$ if and only if $X \cong 0 \in \cC$.
\end{enumerate}
\end{proposition}

\begin{remark}
It is known that $\FPdim(X) \geq 1$ for any object $X \neq 0 \in \cC$; see \cite[Proposition 3.3.4]{EGNO15}.
\end{remark}

\begin{example}\label{eg:FPdim}
\begin{enumerate}
\item In both \cref{eg:fusioncat}\ref{eg:vecG} and \ref{eg:repG}, $\FPdim(X)$ agrees with taking the dimension, $\dim(X)$, of its underlying vector space, as $\dim$ defines a ring homomorphism $\dim\colon K_0(\cC) \ra \R$ satisfying the conditions in \cref{prop:uniquedim}.
More generally, any fusion category with a monoidal, exact, faithful functor $\cF: \cC \to \vec$ (e.g.\ a fiber functor) has $\FPdim(X) = \dim(\cF(X))$.
\item For $\Fib$ from \cref{eg:fusioncat}\ref{eg:Fib}, the relation in \eqref{eqn:goldenratiorel} dictates that $\Pi$ must have $\FPdim(\Pi) = 2\cos(\pi/5)$, i.e. the golden ratio (hence its name).
\end{enumerate}
\end{example}

\subsection{Module categories over fusion categories}
\begin{definition} \label{defn: module category}
A left (resp. right) \emph{module category} over a fusion category $\cC$ is an additive category $\cM$ together with an additive monoidal functor $\cC (\text{resp. } \cC^{\otimes^\text{op}}) \ra \cEnd(\cM)$ into the additive, monoidal category of additive endofunctors $\cEnd(\cM)$.
We shall also say that $\cC$ \emph{acts} on $\cM$.
When $\cM$ is moreover abelian or triangulated, we insist that $\cC$ act on $\cM$ by exact endofunctors.
Given an object $C \in \cC$, the associated (exact) endofunctor will be denoted by $C \otimes -$.
\end{definition}
On the level of the Grothendieck group, a left (resp. right) module category structure on $\cM$ over $\cC$ induces a left (resp. right) module structure on $K_0(\cM)$ over the fusion ring $K_0(\cC)$. The action is given by
\[
[C] \cdot [A] \coloneqq [C \otimes A] \in K_0(\cM).
\]

\begin{example}\label{eg:modulecat}
\begin{enumerate}
\item \label{item:vec-action} Every additive, $\Bbbk$-linear category $\cA$ is naturally a module category over the category of finite dimensional $\Bbbk$-vector spaces $\vec$. 
The action of $\Bbbk^m \in \vec$ is defined on $A \in \cA$ by
$
\Bbbk^m \boxtimes A \coloneqq A^{\oplus m},
$
and every morphism $A \xra{f} A'$ in $\cA$ is sent to
$
\Bbbk^m \boxtimes f \colon A^{\oplus m} \xra{\bigoplus_{i=1}^{m} f} A'^{\oplus m}.
$
Each $n \times m$ matrix $M \in \Hom_{\vec}(\Bbbk^m, \Bbbk^n)$ can be viewed as a morphism from $A^{\oplus m}$ to $A^{\oplus n}$ for each $A \in \cA$ (using linearity of $\cA$), so it defines a natural transformation.
The whole category $\vec$ acts on $\cA$ by some choice of equivalence to its skeleton generated by $\Bbbk$ -- i.e.\ picking a basis.
\item Any fusion category is always a left (and right) module category over itself. Each object $C$ gets sent to the functor $C \otimes -$ (resp. $- \otimes C$).
\item  \label{eg:A4}
Consider the $A_4$ quiver $Q$ with the bipartite orientation.
By viewing $Q$ as the unfolded quiver of the Coxeter quiver
\begin{tikzcd}[column sep = small]
\bullet \ar[r,"5"] & \bullet 
\end{tikzcd},
we obtain an action of $\Fib \cong \TLJ_5^{even}$ on $\rep(Q) \cong \rep(\bullet \xra{5} \bullet)$; see \cite[Example 2.6, 3.2 and 4.12]{heng2023coxeter}.
The exact action of $\Pi \in \Fib$ on the 10 indecomposables of $\rep(Q)$ is indicated on the AR-quiver by orange arrows (see \cref{fig:A4quiver}).
Note that $\1\otimes-$ acts trivially as it is by definition (isomorphic to) the identity functor.
\end{enumerate} 
\end{example}

Some important examples of module categories that we will study in more detail in \cref{sec:Moritadual} come from actions of finite groups (and, as we will see, actions of $\rep(G)$ on its Morita dual).
\begin{definition} \label{defn:Gaction}
Let $G$ be a finite group. Let $\Cat(G)$ be the monoidal category with objects the elements of $G$, the only morphisms are the identities and the tensor product is given by the multiplication of $G$ (so the identity of $G$ is the monoidal unit).
An \emph{action of $G$} on a category $\cM$ is a monoidal functor $T\colon \Cat(G) \ra \cEnd(\cM)$.
As before, if $\cM$ is moreover abelian or triangulated, we will tacitly assume that the functors $T_g\coloneqq T(g)$ are all exact.
\end{definition}

\begin{remark}
\begin{enumerate}
		\item The definition above agrees with the following ``piece-by-piece'' definition that may be more familiar to some readers.
		An action of $G$ consists of an assignment of (invertible) endofunctors $T_g \in \cEnd(\cM)$ for each $g \in G$, together with isomorphisms of functors $\gamma_{g,h}\colon T_g \circ T_{h} \overset{\cong}{\ra} T_{gh}$ for each pair $g,h\in G$, which satisfy $\gamma_{gh,k} \circ \gamma_{g,h} = \gamma_{g,hk} \circ \gamma_{h,k}$.
		The compatibility data are encapsulated by the term ``monoidal functor'' in our definition.

		Note that this definition of a $G$-action (first introduced in \cite{deligneActionGroupeTresses1997}) is more than a group homomorphism $G\rightarrow \cEnd(\cM)$. This finer notion is required to define the $G$-equivariantisation of a category in \cref{defn:equivariantisation}. We refer the reader to \cite[Section 2.2]{beckmannEquivariantDerivedCategories2023} for details on obstructions to lifting such a homomorphism to a $G$-action.
		\item An action of $G$ determines an action of $\vec_G$: We need to define an additive monoidal functor $\Psi\colon \vec_G \rightarrow \cEnd(\cM)$. Since $\vec_G$ is semisimple, it is enough to specify what happens to the simples $g\in G$. Indeed, let $\Psi(g)\coloneqq T_g$, then this extends to any $G$-graded vector space as follows
		\begin{equation*}
			\Psi \colon \oplus_{g\in G} V_g \mapsto \sum_{g\in G} \dim V_g T_g \in \cEnd(\cM).
		\end{equation*}
		The data of the natural isomorphisms $T_g\circ T_h \xrightarrow{\sim}T_{gh}$ then also determines natural isomorphisms $\Psi(X)\otimes \Psi(Y)\xrightarrow{\sim}\Psi(X\otimes Y)$ for any $X,Y\in\vec_G$.
		\item An action of $\vec_G$ determines an action of $G$: the action of $\vec_G$ specifies an endofunctor $T(g)$ for each simple $g\in G$. The rest of the data of a $G$-action can again be determined by the natural isomorphisms of the action of $\vec_G$.
	\end{enumerate}
\end{remark}
\begin{figure}
\centering
$Q \coloneqq$ 
\begin{tikzcd}[column sep=small, row sep=small]
	& 4 \\
	3 \\
	& 2 \\
	1
	\arrow[from=4-1, to=3-2]
	\arrow[from=2-1, to=3-2]
	\arrow[from=2-1, to=1-2]
\end{tikzcd}; 
\qquad
AR quiver of $Q =$
\begin{tikzcd}[column sep=small, row sep=small]
	& {P_1} && {I_4} \\
	{P_2} \arrow[loop above,color={rgb,255:red,214;green,153;blue,92}] && {\tau(S_3)} \arrow[loop above,color={rgb,255:red,214;green,153;blue,92}] && {S_3} \arrow[loop above,color={rgb,255:red,214;green,153;blue,92}]\\
	& {P_3} \arrow[loop below,color={rgb,255:red,214;green,153;blue,92}] && {I_2} \arrow[loop below,color={rgb,255:red,214;green,153;blue,92}]\\
	{P_4} && {\tau(S_1)} && {S_1}
	\arrow[from=2-1, to=1-2]
	\arrow[from=2-1, to=3-2]
	\arrow[from=4-1, to=3-2]
	\arrow[from=1-2, to=2-3]
	\arrow[from=3-2, to=2-3]
	\arrow[from=3-2, to=4-3]
	\arrow[from=2-3, to=1-4]
	\arrow[from=2-3, to=3-4]
	\arrow[from=3-4, to=2-5]
	\arrow[from=1-4, to=2-5]
	\arrow[from=4-3, to=3-4]
	\arrow[from=3-4, to=4-5]
	\arrow[dashed, from=2-5, to=2-3]
	\arrow[dashed, from=2-3, to=2-1]
	\arrow[dashed, from=3-4, to=3-2]
	\arrow[dashed, from=4-5, to=4-3]
	\arrow[dashed, from=4-3, to=4-1]
	\arrow[dashed, from=1-4, to=1-2]
	\arrow[shift right, color={rgb,255:red,214;green,153;blue,92}, from=4-1, to=2-1]
	\arrow[shift right, color={rgb,255:red,214;green,153;blue,92}, from=1-2, to=3-2]
	\arrow[shift right, color={rgb,255:red,214;green,153;blue,92}, from=4-3, to=2-3]
	\arrow[shift right, color={rgb,255:red,214;green,153;blue,92}, from=1-4, to=3-4]
	\arrow[shift right, color={rgb,255:red,214;green,153;blue,92}, from=4-5, to=2-5]
	\arrow[shift right, color={rgb,255:red,214;green,153;blue,92}, from=2-1, to=4-1]
	\arrow[shift right, color={rgb,255:red,214;green,153;blue,92}, from=3-2, to=1-2]
	\arrow[shift right, color={rgb,255:red,214;green,153;blue,92}, from=2-3, to=4-3]
	\arrow[shift right, color={rgb,255:red,214;green,153;blue,92}, from=3-4, to=1-4]
	\arrow[shift right, color={rgb,255:red,214;green,153;blue,92}, from=2-5, to=4-5]
\end{tikzcd}.
\caption{The bipartite $A_4$ quiver (left) and its AR quiver (right). The orange arrows indicate the action of $\Pi \in \Fib$.} \label{fig:A4quiver}
\end{figure}

\section{The closed submanifold of fusion-equivariant stability conditions}
Throughout this section, we will always take ``module category'' to mean ``\emph{left} module category''; the proofs corresponding to right module categories can be obtained through straight-forward modifications.
The notation $\cD$ will always denote a triangulated category.

We begin by recalling the definition of stability conditions and the $G$-invariant version of them, which we shall later generalise to the fusion-equivariant setting
\subsection{Recap: stability functions and stability conditions}
We first recall the notion of stability for abelian categories.
\begin{definition} \label{defn:abelianstability}
A \emph{stability function} on an abelian category $\cA$ is a group homomorphism $Z\in \Hom_{\Z}(K_0(\cA),\C)$ such that for all non-zero objects $E$ in $\cA$,
\[
Z(E)\in \mathbb{H} \cup \R_{<0} = \{re^{i\pi\phi} : r \in \R_{> 0} \text{ and } 0 < \phi \leq 1\}.
\]
Writing $Z(E) = m(E)e^{i\pi\phi}$, we call $\phi$ the \emph{phase} of $E$, which we will denote by $\phi(E)$.
We say that $E$ is $Z$-\emph{semistable} (resp. $Z$-\emph{stable}) if every subobject $A \hookrightarrow E$ satisfies $\phi(A) \leq$ (resp. $<$) $\phi(E)$ . \\
A stability function is said to have the \emph{Harder-Narasimhan (HN) property} if every object $A$ in $\cA$ has a filtration
\[
\begin{tikzcd}[column sep = 3mm]
0
	\ar[rr] &
{}
	{} &
A_1
	\ar[rr] \ar[dl]&
{}
	{} &	
A_2
	\ar[rr] \ar[dl]&
{}
	{} &	
\cdots
	\ar[rr] &
{}
	{} &
A_{m-1}
	\ar[rr] &
{}
	{} &
A_m = A
	\ar[dl] \\
{}
	{} &
E_1
	\ar[lu, dashed] &
{}
	{} &
E_2
	\ar[lu, dashed] &
{}
	{} &
{}
	{} &
{}
	{} &
{}
	{} &
{}
	{} &
E_m
	\ar[lu, dashed] &
\end{tikzcd}
\]
such that all $E_i$ are semistable and $\phi(E_i) > \phi(E_{i+1})$.
Note that this filtration, if it exists, is unique up to isomorphism.
Hence we define $\phi^+(A) \coloneqq \phi(E_1)$ and $\phi^-(A) \coloneqq \phi(E_m)$.
\end{definition}

\begin{definition}
	The \textit{heart of a bounded} $t$\textit{-structure} in a triangulated category $\cD$ is a full additive subcategory $\cA$ such that:
	\begin{enumerate}[(1)]
		\item If $k_1>k_2$ then $\Hom_\cD(\cA[k_1],\cA[k_2])=0$.
		\item For any object $X$ in $\cD$ there are integers $k_1>k_2>\cdots > k_n$, and a sequence of exact triangles:
		\[
\begin{tikzcd}[column sep = 3mm]
0
	\ar[rr] &
{}
	{} &
X_1
	\ar[rr] \ar[dl]&
{}
	{} &	
X_2
	\ar[rr] \ar[dl]&
{}
	{} &	
\cdots
	\ar[rr] &
{}
	{} &
X_{m-1}
	\ar[rr] &
{}
	{} &
X_m = X
	\ar[dl] \\
{}
	{} &
A_1
	\ar[lu, dashed] &
{}
	{} &
A_2
	\ar[lu, dashed] &
{}
	{} &
{}
	{} &
{}
	{} &
{}
	{} &
{}
	{} &
A_m
	\ar[lu, dashed] &
\end{tikzcd}
\]
		such that $A_i\in \cA[k_i]$ for $1\leq i \leq n$.
	\end{enumerate} 
\end{definition}

\begin{definition}
A \emph{slicing} on a triangulated category $\cD$ is a collection of full additive subcategories $\cP(\phi)$ for each $\phi \in \R$ such that
\begin{enumerate}[(1)]
\item $\cP(\phi + 1) = \cP(\phi)[1]$;
\item $\Hom_\cD( \cP(\phi_1), \cP(\phi_2)) = 0$ for all $\phi_1 > \phi_2$; and
\item for each object $X \in \cD$, there exists a filtration of $X$:
\[
\begin{tikzcd}[column sep = 3mm]
0
	\ar[rr] &
{}
	{} &
X_1
	\ar[rr] \ar[dl]&
{}
	{} &	
X_2
	\ar[rr] \ar[dl]&
{}
	{} &	
\cdots
	\ar[rr] &
{}
	{} &
X_{m-1}
	\ar[rr] &
{}
	{} &
X_m = X
	\ar[dl] \\
{}
	{} &
E_1
	\ar[lu, dashed] &
{}
	{} &
E_2
	\ar[lu, dashed] &
{}
	{} &
{}
	{} &
{}
	{} &
{}
	{} &
{}
	{} &
E_m
	\ar[lu, dashed] &
\end{tikzcd}
\]
such that $E_i \in \cP(\phi_i)$ and $\phi_1 > \phi_2 > \cdots > \phi_m$.
\end{enumerate}
Objects in $\cP(\phi)$ are said to be \emph{semistable} with \emph{phase} $\phi$ and the (unique up to isomorphism) filtration in (3) is called the \emph{Harder-Narasimhan (HN) filtration} of $X$.
We will also call the $E_i$ the \emph{HN semistable factors of} $X$.
As before, we shall denote $\phi^+(X) \coloneqq \phi_1$ and $\phi^-(X) \coloneqq \phi_m$.
\end{definition}

Each $\cP(\phi)$ is in fact an abelian category \cite[Lemma 5.2]{bridgeland_2007}. Non-zero simple objects of $\cP(\phi)$ are called $\sigma$-\textit{stable} of phase $\phi$. Given a slicing $\cP$, and interval $I\subset \bbR$, denote by $\cP(I)$ the smallest additive subcategory containing all objects $A$ such that the phases of the semistable factors all live in $I$.
Then $\cP(\phi, \phi+1]$ is the heart of a bounded $t$-structure.
In particular, we call $\cP(0,1]$ the \emph{standard heart} of the slicing $\cP$.
As such, a slicing should be thought of as a refinement of a t-structure.

\begin{definition}[\protect{\cite[Definition 1.1.]{bridgeland_2007}}]
A (Bridgeland) \emph{stability condition} $(\cP, Z)$ on $\cD$ is a slicing $\cP$ and a group homomorphism $Z\in \Hom_{\Z}(K_0(\cD), \C)$, called the \emph{central charge}, satisfying the following condition: for all non-zero semistable objects $E\in \cP(\phi)$,
\[
Z(E) = m(E)e^{i\pi\phi}, \text{ with } m(E) \in \R_{>0}.
\]
The set of stability conditions on $\cD$ is denoted by $\Stab(\cD)$.
\end{definition}
The stability conditions we consider in this paper are required to be \textit{locally-finite}; that is, we require that there exists $\varepsilon > 0$ such that for all $\phi \in \R$, $\cP(\phi - \varepsilon, \phi + \varepsilon)$ is a quasi-abelian category of finite-length; see \cite{Schneiders_quasiabelian} for more details on quasi-abelian categories.
Moreover, if $\cD$ comes with a natural choice of group homomorphism $v\colon K_0(\cD) \ra \Lambda$ to a finite-rank free $\Z$-module $\Lambda$ (a lattice), we shall tacitly assume that all central charges $Z$ factor through $v$.

\begin{example}
	Suppose $\cD$ is Ext-finite and has a Serre functor. Then we can consider the \textit{numerical Grothendieck group}, $\Knum(\cD)$, which is the quotient of $K(\cD)$ by the null-space of the Euler form $\chi(E,F)=\sum_i (-1)^i \dim\left(\Hom(E,F[i])\right)$. $\Knum(\cD)$ is a free $\Z$-module. When $\Knum(\cD)$ has finite rank, we can choose $v\colon K_0(\cD)\ra \Lambda=\Knum(\cD)$ to be the quotient morphism, and consider stability conditions whose central charges factor through $v$. We call such stability conditions \textit{numerical}.

	If $X$ is a smooth projective variety, then $\Knum(X)\coloneqq\Knum(D^b\Coh(X))$ has finite rank by the Riemann--Roch theorem. When $X$ is a surface, $\Knum(X)$ is the image of the Chern character map $\ch\colon K_0(X)\rightarrow H^*(X,\bbQ)$.
\end{example}

\begin{remark}
	It is often assumed that stability conditions satisfy a stronger condition called the \textit{support property} (see \cref{defn: support property}). All of our results hold with this extra assumption, and in one case the proof is simpler; see \cref{section: support property} for details.
\end{remark}

The following criterion by Bridgeland relates stability conditions on $\cD$ and stability functions on its heart (which is abelian).
\begin{proposition}[\protect{\cite[Proposition 5.3]{bridgeland_2007}}] \label{prop:stabonab}
To give a stability condition on $\cD$ is equivalent to giving a heart of a bounded $t$-structure $\cH$ of $\cD$ and a stability function on $\cH$ satisfying the HN property.
\end{proposition}

Bridgeland shows that $\Stab(\cD)$ comes with a natural metric topology, and his main theorem is that $\Stab(\cD)$ is moreover a complex manifold. We remind the reader that all stability conditions were assumed to be locally-finite.
\begin{theorem}[\protect{\cite[Theorem 1.2]{bridgeland_2007}}] \label{thm:stabmfld}
The space of stability conditions on $\cD$, $\Stab(\cD)$, is a complex manifold, with the forgetful map $\cZ$ defining the local homeomorphism 
\begin{align*}
\cZ\colon \Stab(\cD) &\ra \Hom_{\Z}(K_0(\cD),\C) \\
(\cP, Z) &\mapsto Z.
\end{align*}
If moreover $\cD$ comes with a group homomorphism $v\colon K_0(\cD) \ra \Lambda\cong\Z^n$, and all central charges are assumed to factor via $v$, then $\Stab(\cD)$ is a finite-dimensional complex manifold, modelled on $\Hom_{\Z}(\Lambda, \C) \cong \C^n$.
\end{theorem}

Finally, we recall that $\Stab(\cD)$ comes with two commuting continuous actions.
The first of which is by the group of autoequivalences $\Aut(\cD)$, where each $\Phi \in \Aut(\cD)$ acts via:
\[
\Phi \cdot (\cP, Z) = (\cP', Z'), \quad \cP'(\phi) := \Phi(\cP(\phi)), \quad Z'(E) := Z(\Phi^{-1}(E)).
\]
The second of which is by the group $\C$ (under addition), where each $a + ib \in \C$ acts via:
\begin{equation}\label{eq:CactiononStab}
(a+ib) \cdot (\cP, Z) = (\cP', Z'), \quad \cP'(\phi) := \cP(\phi+a), \quad Z'(E) := e^{-(a+ib)i\pi} Z(E).
\end{equation}

\subsection{\texorpdfstring{$G$}{G}-invariant stability conditions}
\begin{definition}\label{defn:Ginvariantstab}
Suppose $\cD$ is equipped with an action of a finite group $G$ (see \cref{defn:Gaction}).
A stability condition $(\cP,Z) \in \Stab(\cD)$ is called \emph{$G$-invariant} if:
\begin{enumerate}
\item $T_g \cP(\phi) = \cP(\phi)$   for all $g \in G$ and for all $\phi \in \R$; and
\item $Z \in \Hom_\bbZ(K_0(\cD), \C)^G \subseteq \Hom_{\Z}(K_0(\cD), \C)$,
\end{enumerate}
where $\Hom_\bbZ(K_0(\cD), \C)^G$ denotes the $\C$-linear subspace of $G$-invariant homomorphisms $Z$, i.e. those satisfying $Z(T_g(X)) = Z(X)$ for all $X \in \cD$.
We use $\Stab_G(\cD) \subseteq \Stab(\cD)$ to denote the subset of all $G$-invariant stability conditions.
\end{definition}

Macr\`i--Mehrotra--Stellari showed that $\Stab_G(\cD)$ is a \emph{closed} complex submanifold:
\begin{theorem}[\protect{\cite[Theorem 1.1]{MMS_09}}] \label{thm:MMSclosedsubmfld}
The subset of $G$-invariant stability conditions $\Stab_G(\cD)$ is a closed complex submanifold of $\Stab(\cD)$.
\end{theorem}

\subsection{Fusion-equivariant stability conditions}
Let us now generalise to the case where $\cD$ is equipped with the structure of a module category over some fusion category $\cC$ (cf. \cref{defn: module category}).
The triangulated $\cC$-module category structure on $\cD$ induces a $K_0(\cC)$-module structure on its Grothendieck group $K_0(\cD)$.
On the other hand, the ring homomorphism induced by the Frobenius--Perron dimension (see \cref{defn:FPdim})
\[
\FPdim\colon K_0(\cC) \ra \R \subset \C,
\]
provides $\C$ with a $K_0(\cC)$-module structure as well.
As such, we can consider the $\C$-linear subspace of $K_0(\cC)$-module homomorphisms inside $\Hom_{\Z}(K_0(\cD), \C)$. These are exactly the homomorphisms $Z\colon K_0(\cC)\rightarrow \C$ such that for all $C\in \cC$ and $[X]\in K_0(\cD)$
\begin{equation} \label{eqn:C-equivariantcentralcharge}
	Z([C]\cdot [X])= Z([C\otimes X])=\FPdim(C) Z([X]).
\end{equation}
We denote the subspace of such homomorphisms by $\Hom_{\Z}(K_0(\cD), \C)$. We shall use these structures to define fusion-equivariant stability conditions as follows.
\begin{definition}\label{defn: C equivariant stab}
Let $\cD$ be a triangulated module category over $\cC$ and let $\sigma = (\cP, Z)$ be a stability condition on $\cD$.
A stability condition $(\cP,Z) \in \Stab(\cD)$ is called \emph{fusion-equivariant} over $\cC$, or \emph{$\cC$-equivariant}, if:
\begin{enumerate}
\item $C \otimes \cP(\phi) \subseteq \cP(\phi)$ for all $C \in \cC$ and all $\phi \in \R$; and
	\label{item:slicingequivariant}
\item $Z \in \Hom_{K_0(\cC)}(K_0(\cD), \C) \subseteq \Hom_{\Z}(K_0(\cD), \C)$, i.e.\ $Z$ satisfies \eqref{eqn:C-equivariantcentralcharge}.
	\label{item:centralchargeequivariant}
\end{enumerate}
We use $\Stab_{\cC}(\cD) \subseteq \Stab(\cD)$ to denote the subset of all $\cC$-equivariant stability conditions.
\end{definition}
Here are some trivial examples:
\begin{example}
\begin{enumerate}
\item With $\cC = \vec$, all stability conditions are $\vec$-equivariant.
\item Let $G$ be a finite group acting on $\cD$; equivalently $\cD$ is a module category over $\vec_G$. 
Then $\vec_G$-equivariant stability conditions are the same as $G$-invariant stability conditions.
\end{enumerate}
\end{example}
Notice that for a fusion-equivariant stability condition, each $\cP(\phi)$ is itself an abelian module category over $\cC$.
It follows that the standard heart $\cH\coloneqq \cP(0,1]$ is an abelian  module category over $\cC$; in fact any $\cP(a,a+1]$ is.
As such, the Grothendieck group $K_0(\cH) \cong K_0(\cD)$ also has a module structure over $K_0(\cC)$.
With the $K_0(\cC)$-module structure on $\C$ as before, we see that the induced stability function $Z\colon K_0(\cH) \ra \C$ on $\cH$ is a $K_0(\cC)$-module homomorphism.
Motivated by this, we make the following definition:
\begin{definition} \label{defn: stab function over C}
Let $\cA$ be an abelian module category over $\cC$.
A stability function $Z\colon K_0(\cA) \ra \C$ is \emph{$\cC$-equivariant} if $Z$ is moreover a $K_0(\cC)$-module homomorphism.
\end{definition}

The following crucial result (first shown in the second-named author's thesis \cite{Heng_PhDthesis}) puts restrictions on the behaviour of semistable objects under the action of fusion category, when our stability function is equivariant and satisfies the HN property. We emphasise here that $C \otimes -$ need not be invertible.

\begin{proposition} \label{prop:cCpreservesemistable}
Let $\cA$ be an abelian module category over $\cC$.
Suppose we have  $Z \in \Hom_{K_0(\cC)}(K_0(\cA),\C) \subseteq \Hom_{\Z}(K_0(\cA),\C)$, a $\cC$-equivariant stability function on $\cA$ that moreover satisfies the Harder--Narasimhan property.
Then for any object $A$ in $\cA$, the following are equivalent:
\begin{enumerate}
\item $A$ is semistable with phase $\phi$; \label{prop:stable}
\item $C \otimes A$ is semistable with phase $\phi$ for all non-zero $C \in \cC$; and \label{prop:allstable}
\item $C \otimes A$ is semistable with phase $\phi$ for some non-zero $C \in \cC$. \label{prop:somestable}
\end{enumerate}
\end{proposition}
\begin{proof}
Given any object $X \in \cA$, we have $\phi^+(X) \geq \phi(X) \geq \phi^-(X)$ by definition (see \cref{defn:abelianstability} for notation). 
Using the group homomorphism property of $Z$, the following are equivalent:
\begin{itemize}
\item $X$ is semistable;
\item $\phi^+(X) = \phi(X)$; and
\item $\phi(X)=\phi^-(X)$.
\end{itemize}
Moreover, the assumption that $Z$ is $\cC$-equivariant dictates that 
\begin{equation}\label{eqn:phasepreserved}
\phi \left(C \otimes X \right) = \phi(X)
\end{equation}
for any object $X \in \cA$ and any $C \in \cC$.
In particular, the phase agreement in the statement of the proposition will be immediate.

The direction \ref{prop:allstable} $\Rightarrow$ \ref{prop:somestable} is tautological, so we will only prove the other two.

Let us start with the easier direction: \ref{prop:somestable} $\Rightarrow$ \ref{prop:stable}.
Suppose $C \otimes A$ is semistable for some $C \neq 0 \in \cC$. Recall that the $\cC$ action is exact. Given an exact sequence in $\cA$,
\[
0 \ra U \ra A \ra V \ra 0,
\]
we can apply the exact functor $C \otimes -$ to obtain the following exact sequence:
\[
0 \ra C \otimes U \ra C\otimes A \ra C \otimes V \ra 0.
\]
The semistability of $C \otimes A$ tells us that
$
\phi(C \otimes U) \leq \phi(C \otimes A).
$
Using \eqref{eqn:phasepreserved} we get
$
\phi(U) \leq \phi(A)
$
as required.

Now let us prove \ref{prop:stable} $\Rightarrow$ \ref{prop:allstable}.
Suppose $A$ is a semistable object.
Let $\Irr(\cC) \coloneqq \{S_j\}_{j = 0}^n$ with $S_0 = \1$ the monoidal unit.
Since every non-zero object in $\cC$ is (up to isomorphism) a direct sum of $S_j$'s, it is sufficient to prove that $S_j \otimes A$ is semistable for all $j$.
Consider the number
\[
r= \max \{\phi^+(S_j \otimes A) \}_{j=0}^n,
\]
so that in particular $r \geq \phi^+(S_0 \otimes A) = \phi^+(A)= \phi(A)$.
If indeed $r = \phi(A)$, then \eqref{eqn:phasepreserved} gives us 
\[
\phi(S_j \otimes A) = \phi(A) = r \geq \phi^+(S_j \otimes A) \geq \phi(S_j \otimes A) \qquad \text{for all } j,
\]
hence all $S_j \otimes A$ are semistable of phase $\phi(A)$ as required.
So assume to the contrary that $r > \phi(A)$. Then for some $S\coloneqq S_k$ we have $r = \phi^+(S \otimes A) > \phi(S \otimes A)$.
In particular, $S \otimes A$ is not semistable. 
Let $E \neq 0$ be the HN semistable factor of $S \otimes A$ with largest phase $\phi(E) = r$, which comes with a non-zero inclusion $E \ra S\otimes A$.
With $\leftdual{S}$ denoting the (left) dual of $S$, the isomorphism  
\[ 
\Hom_{\cA}(E, S \otimes A) \cong \Hom_{\cA}(\leftdual{S}\otimes E, A),  
\]
shows that $\Hom_{\cA}(\leftdual{S}\otimes E, A) \neq 0$.
Since $A$ is semistable and 
\[
\phi(\leftdual{S}\otimes E) = \phi(E) = r > \phi(A),
\] 
it must be the case that $\leftdual{S}\otimes E$ is not semistable.
We then have the HN semistable factor $E'$ of $\leftdual{S} \otimes E$ with the highest phase, giving us
\begin{equation} \label{eqn:E'geqr}
\phi(E') = \phi^+(\leftdual{S} \otimes E) > \phi(\leftdual{S} \otimes E) = r.
\end{equation}
Using exactness of the $\cC$-action, we have found ourselves a chain of inclusions
\begin{equation}\label{eqn:chainofsubs}
E' \ra \leftdual{S} \otimes E \ra \leftdual{S} \otimes S \otimes A.
\end{equation}
With $\cC$ being semisimple, the object $\leftdual{S} \otimes S$ decomposes into a direct sum of simples $S_j$, so we still have $r \geq \phi^+((\leftdual{S} \otimes S) \otimes A)$.
Combined with \eqref{eqn:E'geqr}, we get the strict inequality
\[
\phi(E') > \phi^+( \leftdual{S} \otimes S \otimes A).
\]
Now $E'$ is semistable by construction, so we must have $\Hom_\cA(E', \leftdual{S} \otimes S \otimes A) = 0$.
But this contradicts the chain of inclusions in \eqref{eqn:chainofsubs}, whose composition is non-zero. 
Thus the initial assumption $r> \phi(A)$ must be false, which completes the proof.
\end{proof}
\begin{remark}\label{rem:falsetriangualtedanalogue}
It is important to note that the triangulated category analogue of \cref{prop:cCpreservesemistable} does not hold, namely property \ref{item:slicingequivariant} of a fusion-equivariant stability condition does not follow from property \ref{item:centralchargeequivariant}.
This can already be seen in the case of $\cC = \vec_G$, i.e.\ there are stability conditions whose central charge is $G$-invariant but has some $\cP(\phi)$ not closed under the action of $G$.
\end{remark}

Using this we can show the following variant of \cref{prop:stabonab}.
\begin{theorem} \label{thm: stab function over C and stab cond respecting C}
Giving a $\cC$-equivariant stability condition on $\cD$ is equivalent to giving a heart of a bounded $t$-structure $\cH$ of $\cD$ such that $\cH$ is a module subcategory of $\cD$ over $\cC$ (i.e. $\cC\otimes \cH\subseteq \cH$), together with a $\cC$-equivariant stability function on $\cH$ that satisfies the Harder--Narasimhan property.
\end{theorem}
\begin{proof}
The proof for the implication $(\Rightarrow)$ follows from the discussion before \cref{defn: stab function over C}, whereas the proof for $(\Leftarrow)$ follows from the usual construction of a stability condition from a heart of a bounded $t$-structure with a stability function that satisfies the Harder--Narasimhan property together with \cref{prop:cCpreservesemistable}.
\end{proof}

\subsection{Example: \texorpdfstring{$\Fib$}{Fib} acting on the Fukaya category of the disc with 5 marked points.}
Let $\mathbb{D}_5$ denote the disc with 5 marked points on the boundary, and let $D\Fuk(\mathbb{D}_5)$ denote its Fukaya category, as defined in \cite{HKK_17,OPS_18}.
There is an equivalence
\[
\cD\mathbb{D}_5\coloneqq D\Fuk(\mathbb{D}_5) \cong D^b(\rep(Q)),
\]
where $Q$ is the $A_4$ quiver as in \cref{eg:modulecat}\ref{eg:A4} (see also \cite{QZ_Geometric_model} for types D and E).

The indecomposable objects in $\cD\mathbb{D}_5$ (up to shift) correspond to isotopy classes of curves on $\mathbb{D}_5$ with endpoints on the boundary. 
Note that the $\Fib$-action on $\rep(Q)$ as described in \cref{eg:modulecat}\ref{eg:A4} is exact -- as are all actions of fusion categories on abelian categories. 
As a result $\Fib$ acts on the triangulated category $\cD\mathbb{D}_5$. 

The main result in \cite[Theorem 5.3]{HKK_17} states that the moduli space of flat structures on $\mathbb{D}_5$ is biholomorphic onto its image in $\Stab(\cD\mathbb{D}_5)$.
We describe here a single $\Fib$-equivariant stability condition $\sigma$ on  $\cD\mathbb{D}_5$, defined by the $\Fib$-equivariant stability function $Z$ on the heart $\rep(Q) \subset \cD\mathbb{D}_5$ depicted in \cref{subfig:rootcentralcharge} (refer to \cref{fig:A4quiver} for the action of $\Pi$).
In this case, the image of the stable objects of $\rep(Q)$ under the central charge $Z$ is a copy of the root system associated to the Coxeter group $\mathbb{W}(I_2(5))$, which is the dihedral group (of order 10) of isometries of the pentagon.
This stability condition $\sigma$ (up to the action of $\C$) is also the solution to the Gepner equation $\tau(\sigma) = (-2/5)\cdot \sigma$ studied in \cite{KST_Matrix_fac_quivers, Qiu_Gepner}, where $\tau$ is the Auslander--Reiten translate.

With respect to the flat structure that corresponds to $Z$,  the 5 marked points on the boundary form the 5 vertices of the \emph{regular} pentagon with side lengths one, and under this correpondence the stable objects are the straight lines connecting the 5 marked points; see \cref{subfig:diskpentagon}.
Hence, we have 10 lines connecting the 5 marked points, and 10 corresponding stable objects. 
Notice that the action of $\Pi$ sends an outer edge of length one to the corresponding parallel inner edge of length equal to the golden ratio.

\begin{figure}
\centering
\begin{subfigure}{.45\textwidth}
  \centering
\begin{tikzpicture}[scale = 1.65]
    \coordinate (Origin)   at (0,0);
    \coordinate (XAxisMin) at (-2,0);
    \coordinate (XAxisMax) at (2,0);
    \coordinate (YAxisMin) at (0,-.5);
    \coordinate (YAxisMax) at (0,2);
    \coordinate (na1)      at (-1,0);

    \draw [thin, gray,-latex] (XAxisMin) -- (XAxisMax);
    \draw [thin, gray,-latex] (YAxisMin) -- (YAxisMax);
    \draw[gray, dashed] (1,0) arc (0:180:1);
    \draw[gray, dashed] (1.618,0) arc (0:180:1.618);
    	
	\draw [decorate,decoration={brace,amplitude=10pt}]
		($(Origin) + (0,-0.1)$) -- (-1.618, -0.1) node [black,midway,yshift=-22] 
		{\footnotesize $\delta$};	
	
	\draw[blue, thick, ->] (0,0) -- ++(36:1);
	\node[blue] (s2s1s2a1) at ($(Origin) + (30:1.15)$) 
		{\scriptsize $P_4$};
	
	\draw[blue, thick, ->] (0,0) -- ++(72:1);
	\node[blue] (s2s1a2) at ($(Origin) + (72:1.15)$) 
		{\scriptsize $\tau(S_1)$};
	
	\draw[blue, thick, ->] (0,0) -- ++(108:1);
	\node[blue] (s2a1) at ($(Origin) + (118:1.1)$) 
		{\scriptsize $P_1$};
	
	\draw[blue, thick, ->] (0,0) -- ++(144:1);
	\node[blue] (a2) at ($(Origin) + (154:1.1)$) 
		{\scriptsize $I_4$};
		
	\draw[blue, thick, ->] (0,0) -- ++(180:1);
	\node[blue] (a1) at ($(Origin) + (186:1.15)$) 
		{\scriptsize $S_1$};
	
	\draw[red, ->] (0,0) -- ++(36 :1.618);
	\node[red] (xa1) at ($(Origin) + (36:1.8)$) 
		{\scriptsize $P_2$};
	
	\draw[red, ->] (0,0) -- ++(72 :1.618);
	\node[red] (xa1) at ($(Origin) + (68:1.8)$) 
		{\scriptsize $\tau(S_3)$};
	
	\draw[red, ->] (0,0) -- ++(108:1.618);
	\node[red] (xa1) at ($(Origin) + (108:1.8)$) 
		{\scriptsize $P_3$};
	
	\draw[red, ->] (0,0) -- ++(144:1.618);
	\node[red] (xa2) at ($(Origin) + (144:1.8)$) 
		{\scriptsize $I_2$};
		
	\draw[red, ->] (0,0) -- ++(180  :1.618);
	\node[red] (xa1) at ($(Origin) + (-184:1.8)$) 
		{\scriptsize $S_3$};
    
    \pic [draw, ->, "$\frac{\pi}{5}$", angle eccentricity=1.7] {angle=xa2--Origin--na1};
  \end{tikzpicture}
  \caption{
  		The central charge $Z$.
  }
  \label{subfig:rootcentralcharge}
\end{subfigure}   
\begin{subfigure}{.45\textwidth}
  \centering
\begin{tikzpicture}
  \draw[dotted] (0,0) circle [radius=2.15];
  \foreach \i in {1,...,5}
    \fill (\i*360/5:2.1) coordinate (n\i) circle(2 pt)
      \ifnum \i=5 
      	(n1) edge node[fill=gray!30] {\tiny $S_1$} (n2)
      	(n1) edge node[fill=gray!30] {\tiny $P_2$} (n3) 
      	(n1) edge node[fill=gray!30] {\tiny $P_3$} (n4) 
      	(n1) edge node[fill=gray!30, scale=0.8] {\tiny $\tau(S_1)$} (n5) 
      	(n2) edge node[fill=gray!30] {\tiny $P_1$} (n3) 
      	(n2) edge node[fill=gray!30, scale=0.8] {\tiny $\tau(S_3)$} (n4) 
      	(n2) edge node[fill=gray!30] {\tiny $I_2$} (n5) 
      	(n3) edge node[fill=gray!30] {\tiny $I_4$} (n4) 
      	(n3) edge node[fill=gray!30] {\tiny $S_3$} (n5) 
      	(n4) edge node[fill=gray!30] {\tiny $P_4$} (n5)
      	\fi;
\end{tikzpicture}
  \caption{
  		The flat disk $\mathbb{D}_5$.
  }
  \label{subfig:diskpentagon}
\end{subfigure}%
 \caption{
  		The images of the central charge $Z$ for all of the indecomposable objects in $\rep(Q)$ are shown on the left, where the blue (shorter) lines have length 1 and the red (longer) lines have length $\delta=$ golden ratio. The flat structure on the disk $\mathbb{D}_5$ on the right is induced by the stability condition defined by $Z$. 
  }
  \label{fig:A4rootdisk}
\end{figure}

\subsection{Submanifold property}
Recall that $\Stab_{\cC}(\cD) \subseteq \Stab(\cD)$ denotes the subset of all $\cC$-equivariant stability conditions.
By definition, the forgetful map $\cZ$ above maps $\sigma \in \Stab_{\cC}(\cD)$ into the $\C$-linear subspace $\Hom_{K_0(\cC)}(K_0(\cD), \C) \subseteq \Hom_{\Z}(K_0(\cD), \C)$ of $K_0(\cC)$-module morphisms.
To show that $\Stab_{\cC}(\cD)$ is a complex submanifold, it will be sufficient to show that any local deformation of $Z = \cZ(\sigma)$ within the linear subspace $\Hom_{K_0(\cC)}(K_0(\cD),\C)$ lifts to a stability condition $\tau$ that is still $\cC$-equivariant.
To do so, we first recall the following deformation result by Bridgeland.
\begin{lemma}[\protect{\cite[Theorem 7.1]{bridgeland_2007}}]\label{lem:deformation}
Let $\sigma = (\cP, Z)$ be a stability condition. 
There exists $\varepsilon_0 >0$ such that for all $0<\varepsilon< \varepsilon_0$, if $W \in \Hom_{\Z}(K_0(\cD),\C)$ satisfies
\[
|W(E) - Z(E)|<\sin(\varepsilon \pi)|Z(E)|
\]
for all $\sigma$-semistable object $E$, then there exists a slicing $\cQ$ so that $\tau = (\cQ, W)$ forms a stability condition on $\cD$ with $d(\sigma,\tau) < \varepsilon$, where $d$ is the distance defined by
\[
d(\sigma,\tau)\coloneqq \sup_{0\neq E \in \cD} \left\{
	|\phi^-_\cP(E) - \phi^-_\cQ(E)|, |\phi^+_\cP(E) - \phi^+_\cQ(E)|, |\log m_\sigma(E) - \log m_\tau(E)|
	\right\}.
\]
Moreover, the slicing $\cQ$ is constructed by defining each $\cQ(\psi)$ to be the full additive subcategory of $\cD$ containing the zero object and all objects $E$ that are $W$-semistable with phase $\psi$ in any thin subcategory $\cP(a,b)$ such that $a+\varepsilon\leq \psi \leq b-\varepsilon$.
\end{lemma}
We shall strengthen the lemma in the case where the stability condition $\sigma$ is $\cC$-equivariant and $W$ is a $K_0(\cC)$-module homomorphism.
Before we do so, we remind the reader of some of the salient properties of $\cP(a,b)$ and quasi-abelian categories that we require in the proof, which can all be found in \cite[Section 4]{bridgeland_2007}. 
With $\varepsilon > 0$ fixed, recall that $\cP(a,b)$ is said to be \emph{thin} if $0 < b - a < 1 - 2\varepsilon$.
\begin{itemize}
\item Suppose $\cP(a,b)$ is thin. By \cite[Lemma 4.3]{bridgeland_2007}, $\cP(a,b)$ is quasi-abelian and strict exact sequences $0 \ra X \xra{f} Y \xra{g} Z \ra 0$ in $\cP(a,b)$ are in one-to-one correspondence with exact triangles $X \xra{f} Y \xra{g} Z \xra{h} X[1]$ in $\cD$ with $X,Y$ and $Z$ all in $\cP(a,b)$. In this case, the morphism $f: X \ra Y$ is a strict monomorphism in $\cP(a,b)$.
\item If $f: X \ra Y$ is a strict monomorphism and $X \neq 0$, then $f \neq 0$.
\item Strict monomorphisms are closed under compositions.
\end{itemize}
\begin{lemma}\label{lem: openness of fusion-equivariant stability}
Suppose further that $\sigma = (\cP,Z)$ in \cref{lem:deformation} is $\cC$-equivariant. 
Then for any $W$ which is moreover in $\Hom_{K_0(\cC)}(K_0(\cD), \C)$, the lifted stability condition $\tau=(\cQ, W)$ is also $\cC$-equivariant.
\end{lemma}
\begin{proof}
By assumption, $W$ is in $\Hom_{K_0(\cC)}(K_0(\cD), \C)$, so it is sufficient to show that $A \in \cQ(\psi)$ implies that $C \otimes A \in \cQ(\psi)$ for any $C \in \cC$.

We claim that a proof along the same line as the proof of \ref{prop:stable} $\Rightarrow$ \ref{prop:allstable} in \cref{prop:cCpreservesemistable} almost works.
Namely, we work in a triangulated category $\cD$ instead of an abelian category $\cA$, and we replace ``semistable'' with ``$\tau$-semistable'' so that all HN filtrations are taken with respect to the slicing $\cQ$.
Note in particular that any inclusion that appeared in the proof is now merely a map in $D$.
One can check that all of the arguments go through until we reach the final contradiction required. Indeed, the composition in \eqref{eqn:chainofsubs} could very well be zero in this setting\footnote{
As pointed out in \cref{rem:falsetriangualtedanalogue}, the triangulated analogue of \cref{prop:cCpreservesemistable} is false, so one can not expect the same proof to work for triangulated categories without extra assumptions.
} 
(previously this was a composition of inclusions in an abelian category, which must be non-zero as $E'$ is non-zero).
Nonetheless, the stability condition $\tau = (\cQ, W)$ is (by assumption) within $\varepsilon$-distance from a $\cC$-equivariant stability condition $\sigma = (\cP, Z)$.
We will leverage this assumption and use it to show that the composition in \eqref{eqn:chainofsubs} is a composition of strict monomorphisms in the thin (this property holds since $\varepsilon_0$ can be chosen to be small), hence quasi-abelian, subcategory 
\[
\cP(a,b) \coloneqq \cP(\psi-3\varepsilon, \psi+5\varepsilon),
\]
which is therefore non-zero since $E'$ is non-zero.
The rest of the proof is devoted to proving this statement.

The notations used here will be the same as those in \cref{prop:cCpreservesemistable}.
Throughout the proof, we will repeatedly use the fact that
\[
X \in \cP(x,y) \implies X \in \cQ(x-\varepsilon, y+\varepsilon) 
\]
and similarly
\[
X \in \cQ(x,y) \implies X \in \cP(x-\varepsilon, y+\varepsilon),
\]
which are both consequences of $d(\sigma,\tau) < \varepsilon$.
Since $A \in \cQ(\psi)$, $A$ lives in the quasi-abelian subcategory $\cP(\psi-\varepsilon, \psi+\varepsilon)$.
The assumption that $(\cP,Z)$ is $\cC$-equivariant then guarantees that $S \otimes A$ is also in $\cP(\psi - \varepsilon, \psi + \varepsilon)$.
Consequently, $S \otimes A \in \cQ(\psi - 2\varepsilon, \psi + 2\varepsilon)$.
Let $E$ be the $\tau$-HN semistable factor of $S\otimes A$ of highest phase and $E \rightarrow S\otimes A$ the corresponding morphism. Put $F = \Cone(E\rightarrow S\otimes A)$ and consider the exact triangle:
\begin{equation} \label{eqn:extrislicing}
E \ra S \otimes A \ra F \ra \ .
\end{equation}
By construction, the set of HN semistable factors of $F$ is obtained from that of $S\otimes A$ by removing $E$. Therefore we have $F \in \cQ(\psi-2\varepsilon, \psi+2\varepsilon)$, which implies that $F \in \cP(\psi-3\varepsilon, \psi+3\varepsilon) \subset \cP(a,b)$.
Moreover, we have $E \in \cQ(r)$ with $\psi < r < \psi + 2\varepsilon$, and so $E \in \cP(\psi - \varepsilon, \psi + 3\varepsilon) \subset \cP(a,b)$.
In particular, all three objects in \eqref{eqn:extrislicing} are in $\cP(a,b)$, so $E \ra S \otimes A$ is a strict monomorphism in $\cP(a,b)$.
Applying the functor $\leftdual{S} \otimes - $ to \eqref{eqn:extrislicing}, we get an exact triangle with all three objects in $\cP(a,b)$ ($\sigma$ is $\cC$-equivariant), hence $\leftdual{S} \otimes E \ra \leftdual{S} \otimes S \otimes A$ is also a strict monomorphism in $\cP(a,b)$.

The proof that $E' \ra \leftdual{S} \otimes E$ is a strict monomorphism follows similarly, starting with the fact that $\leftdual{S}\otimes E \in \cP(\psi - \varepsilon, \psi + 3\varepsilon)$.
\end{proof}

\begin{theorem}\label{thm:equivariantissubmfld}
The space of $\cC$-equivariant stability conditions, $\Stab_\cC(\cD)$ is a complex submanifold of $\Stab(\cD)$ with the restriction of the forgetful map $\cZ$ (from \cref{thm:stabmfld})
\begin{align*}
\cZ\colon \Stab_\cC(\cD) &\ra \Hom_{K_0(\cC)}(K_0(\cD),\C) \subseteq  \Hom_{\Z}(K_0(\cD),\C) \\
(\cP, Z) &\mapsto Z
\end{align*} 
providing the local homeomorphism.
If moreover $\cD$ comes with $v\colon K_0(\cD) \ra \Lambda$, so that all central charges are assumed to factor via $v$, then $\Stab_{\cC}(\cD)$ is an $m$-dimensional submanifold with $m \leq n = \rank(\Lambda)$.
\end{theorem}
\begin{proof}
This is a direct consequence of \cref{lem: openness of fusion-equivariant stability} together with \cref{thm:stabmfld}.
\end{proof}
\begin{remark}
\cref{thm:equivariantissubmfld} was obtained independently in \cite{QZ_fusion-stable}.
\end{remark}

\subsection{Closed property}\label{section: closed property}
In this subsection, we show that $\Stab_{\cC}(\cD)$ is moreover a closed subset of $\Stab(\cD)$.
The results here are highly motivated by -- and will make use of -- the constructions and results in \cite{MMS_09}. We shall now recall the relevant ones.

Let $\Phi\colon \cD \ra \cD'$ be an exact functor between triangulated categories satisfying the following property, called  (Ind) in \cite{MMS_09}:
\begin{equation}\tag{P1} \label{assump:nonzero}
\Hom_{\cD'}(\Phi(X), \Phi(Y)) = 0 \implies \Hom_{\cD}(X,Y) = 0 \quad \text{ for all } X, Y \in \cD.
\end{equation}
Given $\sigma' = (\cP', Z') \in \Stab(\cD')$, we define for each $\phi \in \R$ the following additive subcategory of $\cD$:
\[
\Phi^{-1}\cP'(\phi) \coloneqq \{ E \in \cD \mid \Phi(E) \in \cP'(\phi) \}
\]
and a group homomorphism $\Phi^{-1}Z'\colon K_0(\cD) \ra \C$ defined by
\[
[X] \mapsto Z'[\Phi(X)].
\]
The pair $\Phi^{-1}\sigma'\coloneqq (\Phi^{-1}\cP', \Phi^{-1}Z')$ might not always define a stability condition on $\cD$ -- one has to check if the Harder--Narasimhan property is satisfied.
Nonetheless, if we restrict to those that do, we have the following result.
\begin{lemma}[\protect{\cite[Lemma 2.8, 2.9]{MMS_09}}] \label{lem:domclosedcts}
Suppose $\Phi$ satisfies property \eqref{assump:nonzero}.
Then the following subset of $\Stab(\cD')$
\[
\Dom(\Phi^{-1}) \coloneqq \{ \sigma' \in \Stab(\cD') \mid \Phi^{-1}\sigma' \in \Stab(\cD) \}
\]
is closed.
Moreover, the map from this subset
\[
\Phi^{-1}\colon \Dom(\Phi^{-1}) \ra \Stab(\cD)
\]
is continuous.
\end{lemma}

Now let us return our focus to triangulated module categories $\cD$ over $\cC$.
\begin{lemma} \label{lem:Cactionpullback}
Let $C \in \cC$.
Then
\begin{enumerate}
\item the endofunctor $C \otimes -\colon \cD \ra \cD$ satisfies \eqref{assump:nonzero};
\item for all $\cC$-equivariant stability conditions $\sigma = (\cP, Z) \in \Stab_{\cC}(\cD)$, we have
\[
(C\otimes -)^{-1}\cP = \cP \quad \text{and} \quad (C\otimes -)^{-1}Z = \FPdim(C)\cdot Z,
\]
so that $(C\otimes -)^{-1}\sigma$ is also a $\cC$-equivariant stability condition. 
\end{enumerate}
In particular, $\Stab_{\cC}(\cD) \subseteq \Dom((C\otimes -)^{-1})$ and $(C \otimes -)^{-1}$ maps $\Stab_{\cC}(\cD)$ to itself.
\end{lemma}
\begin{proof}
To prove (i), note that the existence of duals for $C$ gives the following natural isomorphism
\[
\Hom_{\cD}(C\otimes X, C\otimes Y) \cong \Hom_{\cD}(X, \rightdual{C} \otimes C \otimes Y).
\]
By semisimplicity and duality, the monoidal unit $\1$ must appear as a summand of $\rightdual{C} \otimes C$, so $\Hom_{\cD}(X, Y)$ appears as a summand of $\Hom_{\cD}(X, \rightdual{C} \otimes C \otimes Y)$.
This proves that \eqref{assump:nonzero} is satisfied by $C\otimes -$.

Let us now prove (ii).
Suppose $\sigma = (\cP, Z)$ is $\cC$-equivariant.
Then it follows directly from the $\cC$-equivariant property that $(C\otimes -)^{-1}Z[X] = Z[C\otimes X] = \FPdim(C)\cdot Z[X]$ for all $[X] \in K_0(\cD)$.
We are left to show that $(C\otimes -)^{-1}\cP = \cP$. 
The containment $(C\otimes -)^{-1}\cP \supseteq \cP$ is direct.
Indeed, given $X \in \cP(\phi)$, $\sigma$ being $\cC$-equivariant says that $C \otimes X \in \cP(\phi)$, which gives us $X \in (C\otimes -)^{-1}\cP(\phi)$.
Conversely, let $X \in (C\otimes -)^{-1}\cP(\phi)$, so that we know $C\otimes X \in \cP(\phi)$.
Applying \cref{prop:cCpreservesemistable} to any preferred choice of heart containing $\cP(\phi)$ shows that $X$ must also be semistable of the same phase, i.e. $X \in \cP(\phi)$ as required.

Note that $(C\otimes -)^{-1}\sigma = (\cP, \FPdim(C) \cdot Z)$ is indeed a stability condition, since $(\cP, \FPdim(C) \cdot Z)$ is obtained from the (continuous) action of $\C$ on $\sigma \in \Stab(\cD)$, i.e.\ $i \log \FPdim(C)/\pi \in \C$ applied to $\sigma$.
The fact that it is $\cC$-equivariant is immediate.

The final statement in the lemma is a direct consequence of (i) and (ii).
\end{proof}

Now we are ready to prove that $\Stab_{\cC}(\cD)$ is closed.
\begin{theorem}\label{thm:equivariantisclosed}
The subset $\Stab_{\cC}(\cD) \subseteq \Stab(\cD)$ of $\cC$-equivariant stability conditions is a closed subset.
\end{theorem}
\begin{proof}
Let us consider the set
\[
\Delta \coloneqq \{ \left( (\FPdim(S)\cdot Z, \cP) \right)_{S \in \Irr(\cC)} \} \subseteq \prod_{S \in \Irr(\cC)} \Stab(\cD),
\]
which is a closed subset.
Indeed, the diagonal is closed in $\prod_{S \in \Irr(\cC)} \Stab(\cD)$ and $\Delta$ is obtained by the continuous action of $\C$ on the appropriate components of $\Stab(\cD)$ (acted on by $i \log \FPdim(S)/\pi \in \C$).
By \cref{lem:domclosedcts}, we get that $F \coloneqq \bigcap_{S \in \Irr(\cC)} \Dom((S\otimes -)^{-1})$ is a closed subset of $\Stab(\cD)$.
Moreover, the map
\[
\Psi \coloneqq \prod_{S\in \Irr(\cC)} (S\otimes -)^{-1}|_F\colon F \ra \prod_{S \in \Irr(\cC)} \Stab(\cD)
\]
is continuous.
We denote by
\[
I \coloneqq \Psi^{-1}(\Delta)
\]
the inverse image of $\Delta$ under $\Psi$, which is a closed subset of $F$.
Since $F$ itself is closed in $\Stab(\cD)$, $I$ is also a closed subset of $\Stab(\cD)$.

We claim that $I$ is exactly $\Stab_{\cC}(\cD)$, which would then complete the proof.
\cref{lem:Cactionpullback} shows that $\Stab_{\cC}(\cD)$ is contained in $I$, so it remains to show the converse.
Suppose $\sigma = (\cP, Z) \in I$, which means for all $S \in \Irr(\cC)$, we have $(S\otimes -)^{-1}\sigma = (\cP, \FPdim(S)\cdot Z)$.
Since $\FPdim$ is a ring homomorphism and $\cC$ is semisimple, we get that this is actually true for all $C \in \cC$, namely
\[
(C \otimes -)^{-1}\sigma = (\cP, \FPdim(C)\cdot Z) \quad \text{for all } C \in \cC.
\]
Having $(C \otimes -)^{-1}\cP(\phi) = \cP(\phi)$ for all $C \in \cC$ and all $\phi \in \R$ implies that $C \otimes \cP(\phi) \subseteq \cP(\phi)$ for all $C$ and all $\phi$.
On the other hand, $(C\otimes -)^{-1}Z = \FPdim(C)\cdot Z$ shows that $Z$ is a $K_0(\cC)$-module homomorphism.
As such, $\sigma$ is indeed $\cC$-equivariant.
\end{proof}

Let us record the main result of this whole section, which combines \cref{thm:equivariantissubmfld} and \cref{thm:equivariantisclosed}; and generalises \cref{thm:MMSclosedsubmfld} (i.e.\ \cite[Theorem 1.1]{MMS_09}).
\begin{corollary}\label{cor: C-equivariant complex submanifold}
Let $\Stab_{\cC}(\cD) \subseteq \Stab(\cD)$ be the subset of $\cC$-equivariant stability conditions on the triangulated module category $\cD$ over $\cC$.
Then $\Stab_{\cC}(\cD)$ is a closed, complex submanifold of $\Stab(\cD)$ that is locally modelled on the subspace $\Hom_{K_0(\cC)}(K_0(D), \bbC) \subseteq \Hom_\bbZ(K_0(D), \bbC)$.
\end{corollary}

\section{A duality for stability conditions} \label{sec:Moritadual}
Throughout this section, we let $G$ be a finite group and $\cD$ be a $\Bbbk$-linear triangulated module category over $\vec_G$; we remind the reader that this is equivalent to saying $\cD$ is equipped with an action of $G$ (in the sense of \cref{defn:Gaction}). To ensure that $\rep(G)$ is a fusion category, we further assume throughout that the underlying field $\Bbbk$ has characteristic not dividing $|G|$.
We shall also use the convention that module categories mean left module categories.
The main result of this section is to show (under mild assumptions on $\cD$) that the space of $\vec_G$-equivariant (equivalently, $G$-invariant) stability conditions is biholomorphic to the space of $\rep(G)$-equivariant stability conditions
\[
\Stab_{\vec_G}(\cD) \cong \Stab_{\rep(G)}(\cD^G),
\]
where  $\cD^G$ is the $G$-equivariantisation of $\cD$, which we shall define in the next section.

\subsection{Morita duality: \texorpdfstring{$G$}{G} (de-)equivariantisation}
\begin{definition}\label{defn:equivariantisation}
Let $\cD$ be a $\Bbbk$-linear triangulated module category over $\vec_G$.
We define the \emph{$G$-equivariantisation} of $\cD$, denoted by $\cD^G$, as follows:
\begin{itemize}
\item the objects, called the \emph{$G$-equivariant objects}, are the pairs $(X, (\theta_g)_{g\in G})$ with $X \in \cD$ and isomorphisms $\theta_g\colon g \otimes X \ra X$ in $\cD$ making the following diagrams commutative:
\[
\begin{tikzcd}
g \otimes h \otimes X \ar[r,"g \otimes \theta_h"] \ar[d, "\mult\otimes \id_X", swap] & g \otimes X \ar[d, "\theta_g"] \\
gh \otimes X \ar[r, "\theta_{gh}"] & X
\end{tikzcd};
\]
\item the morphisms are those in $\cD$ that commute with the isomorphisms $\theta_g$, $g\in G$.
\end{itemize}
\end{definition}
The category $\cD^G$ is a $\Bbbk$-linear module category over $\rep(G)$; see e.g.\ \cite[\S 4.1.3]{DGNO_braidedI}.
Explicitly, recall from \cref{eg:modulecat}\ref{item:vec-action} that any $\Bbbk$-linear category is a module category over $\vec$.
Through this structure, each representation $(V, (\rho_g)_{g\in G}) \in \rep(G)$ acts on $(X, (\theta_g)_{g \in G}) \in \cD^G$ by
\[
(V, (\rho_g)_{g\in G}) \otimes (X, (\theta_g)_{g \in G})\coloneqq
	(V \boxtimes X, (\theta^\rho_g)_{g \in G}),
\]
where each $\theta^\rho_g$ is defined via the composition
\[
\theta^\rho_g: g \otimes (V \boxtimes X) \xra{\cong} V\boxtimes (g \otimes X) \xra{\rho_g \boxtimes \theta_g} V \boxtimes X.
\]
The first isomorphism above comes from the fact that $G$ acts via additive functors.
The action of $(V, (\rho_g)_{g\in G})$ on morphisms is defined via the action of $V$ on $\cD^G$.
We leave it to the reader to check the required commutativity conditions.

In general, $\cD^G$ need not be a triangulated category. However, we will only consider cases where $\cD^G$ forms a triangulated category in the following sense.

\begin{assumption}\label{assump:triang}
Let $\cF: \cD^G \ra \cD$ denote the forgetful functor sending $(X, (\theta_g)_{g\in G}) \mapsto X$. We assume $\cD^G$ is triangulated in such a way that $\cF$ preserves and reflects exact triangles.
In other words, 
\[
(X, (\theta_g)_{g \in G}) \ra (Y, (\theta'_g)_{g \in G}) \ra (Z, (\theta''_g)_{g \in G}) \ra \text{ is exact } \iff X \ra Y \ra Z \ra \text{ is exact.}
\]
\end{assumption}

Note that the assumption above is satisfied in many situations of interest:
\begin{itemize}
\item Let $\cA$ be an abelian module category over $\vec_G$. Then $D^b\cA$ is a triangulated module category over $\vec_G$. Moreover, $(D^b\cA)^G\cong D^b(\cA^G)$  by \cite[Theorem 7.1]{elaginEquivariantTriangulatedCategories2015} and it follows that $\cD = D^b\cA$ satisfies \cref{assump:triang}.
\item Let $\cD$ be a triangulated module category over $\vec_G$ and suppose $\cD$ has a DG-enhancement (the $G$-action need not lift to the DG-setting). Then by \cite[Corollary 6.10]{elaginEquivariantTriangulatedCategories2015}, $\cD$ satisfies \cref{assump:triang}.
\end{itemize}

\begin{remark}
Assume $\characteristic (\Bbbk) =0$. Let $\vec_G\lMOD$ and $\rep(G)\lMOD$ denote the 2-category of $\Bbbk$-linear, idempotent complete (additive) module categories over $\vec_G$ and $\rep(G)$ respectively.
The 2-functor sending $\cA \in \vec_G\lMOD$ to its $G$-equivariantisation $\cA^G \in \rep(G)\lMOD$ is a 2-equivalence \cite[Theorem 4.4]{DGNO_braidedI}.
This is known as the (1-)categorical Morita duality of $\vec_G$ and $\rep(G)$.
\end{remark}

\begin{remark}
Suppose $G$ is moreover abelian, so that $\rep(G) \cong \vec_{\widehat{G}}$, where $\widehat{G}\coloneqq\Hom(G,\Bbbk^\ast)$ is the Pontryagin dual. Then $(\cA^G)^{\widehat{G}}\cong\cA$ by \cite[Theorem 4.2]{elaginEquivariantTriangulatedCategories2015}.
\end{remark}

\subsection{\texorpdfstring{$G$}{G}-invariant and \texorpdfstring{$\rep(G)$}{rep(G)}-equivariant stability conditions}

In what follows, we let
 $\cD$ be a triangulated module category over $\vec_G$ whose equivariantisation  $\cD^G$ satisfies \cref{assump:triang}.

To begin, let us record some auxiliary results.
\begin{lemma}	\label{lem:FIproperties}
We have the following:
\begin{enumerate}
\item The forgetful functor $\cF\colon \cD^G \ra \cD$ is faithful.
\item Consider the induction functor $\cI\colon \cD \ra \cD^G$ defined by
\[
\cI(X) \coloneqq \left( 
	\bigoplus_{g \in G} (g \otimes X), (\wt{\epsilon}_h)_{h \in G}
	\right),
\]
where $\wt{\epsilon}_h$ is defined on each summand by the isomorphism 
\[
h \otimes h^{-1}g \otimes X \xra{\mult \otimes \id_X} g \otimes X.
\]
Then $\cI$ is a biadjoint of the forgetful functor $\cF$.
Moreover, $\cI$ is exact. \label{lem:ind}
\item The composite functor $\cI \circ \cF\colon \cD^G \ra \cD^G$ is equivalent to the functor $R\otimes -$ given by the action of the regular representation $R \coloneqq \Bbbk[G]\in \rep(G)$ on $\cD^G$.
\item The composite functor $\cF \circ \cI\colon \cD \ra \cD$ sends objects $X \mapsto \bigoplus_{g \in G} g \otimes X$.
\end{enumerate}
\end{lemma}
\begin{proof}
The fact that $\cI$ is exact follows from the exactness of the action of $\vec_G$ (which is part of the definition of a module structure) and \cref{assump:triang} on $\cD$.

The rest follow from a standard calculation; see e.g.\ \cite{Dem_11, DGNO_braidedI} for the case of abelian categories.
\end{proof}
\begin{lemma}\label{lem:FIP1}
Both the functors $\cI: \cD \ra \cD^G$ and $\cF: \cD^G \ra \cD$ satisfy property \eqref{assump:nonzero}.
\end{lemma}
\begin{proof}
Using the properties of $\cI$ and $\cF$ from \cref{lem:FIproperties}, we get that
\[
\Hom_{\cD^G}(\cI(X), \cI(Y)) \cong \Hom_{\cD}(X,(\cF\circ\cI)(Y)) \cong \Hom_{\cD}\left(X, \bigoplus_{g \in G} g\otimes Y \right).
\]
The right-most hom space contains $\Hom_{\cD}(X,Y)$ as a summand (pick $g$ to be the identity element of $G$), so $\cI$ satisfies property \eqref{assump:nonzero}.
The same argument can be applied to $\cF$, or one could use the fact that $\cF$ is faithful.
\end{proof}
Before we state our main theorem, we will need the following assumptions on $\cD$.

\begin{assumption}\label{assumptions: bigger categories with small coproducts}
	\begin{enumerate}
	\item $\cD$ is essentially small and satisfies \cref{assump:triang}.
	\item The $\vec_G$ action on $\cD$ extends to an action on a triangulated category 
	$\wt{\cD}$ (containing $\cD$), which contains all small coproducts and satisfies  \cref{assump:triang}.
	\end{enumerate}
\end{assumption}

For example, let $\cA$ be an essentially small abelian category with a $G$-action that extends to an abelian category $\wt{\cA}$ containing all small coproducts. 
This includes: 
\begin{itemize}
\item $\cA = A\lmod$ and $\wt{\cA} = A\lMod$; here $A$ is a finite-dimensional algebra over $\Bbbk$ equipped with a $G$-action, and $A\lmod$ (resp.\ $A\lMod$) denotes the category of finite-dimensional (resp.\ all) modules; and
\item $\cA = \Coh(X)$ and $\wt{\cA} = \QCoh(X)$, where $X$ is a scheme over $\Bbbk$ equipped with a $G$-action.
\end{itemize}
Then the assumptions above hold for $\cD = D^b\cA$ (together with its induced $G$-action) with $\wt{\cD} = D\wt{\cA}$ the unbounded derived category on $\wt{\cA}$.
\begin{theorem}\label{thm:moritastability}
Let $\cD$ be a triangulated module category over $\vec_G$ satisfying \cref{assumptions: bigger categories with small coproducts}.
Then the induction functor $\cI\colon \cD \ra \cD^G$ and the forgetful functor $\cF\colon \cD^G \ra \cD$ induce biholomorphisms between the closed submanifolds of stability conditions $\Stab_{\rep(G)}(\cD^G)$ and $\Stab_{\vec_G}(\cD)$,
	\[
		\cI^{-1}\colon \Stab_{\rep(G)}(\cD^G) \overset{\cong}{\rightleftarrows} \Stab_{\vec_G}(\cD) \cocolon\cF^{-1},
	\]
	which are mutually inverse up to rescaling the central charge by $|G|$.
\end{theorem}
\begin{proof}
We first show that $\Dom(\cI^{-1})$ and $\Dom(\cF^{-1})$ contain $\Stab_{\rep(G)}(\cD^G)$ and $\Stab_{\vec_G}(\cD)$ respectively. 
This will be achieved through an application of \cite[Theorem 2.14]{MMS_09}, which is a specific case of the more general result \cite[Theorem 2.1.2]{Pol_constant-t}.
To this end, note that \cref{assumptions: bigger categories with small coproducts} on $\cD$ gives us triangulated categories $\wt{\cD} \supseteq \cD$ and $\wt{\cD}^G \supseteq \cD^G$, with $\wt{\cD}$ satisfying \cref{assump:triang}.
As such, \cref{lem:FIproperties} also applies to $\wt{\cD}$.
It is immediate that the forgetful functor $\wt{\cF}$ and induction functor $\wt{\cI}$ between $\wt{\cD}$ and $\wt{\cD}^G$ restrict to the usual forgetful functor $\cF$ and induction functor $\cI$ between $\cD$ and $\cD^G$.
Note that $\wt{\cF}(E) \in \cD$ implies that $E \in \cD^G$, since $(\wt{\cI}\circ \wt{\cF})(E) = \Bbbk[G] \otimes E$ and $\cD^G$ is closed under the action of $\rep(G)$; similarly $\wt{\cI}(E') \in \cD^G$ implies $E' \in \cD$.
Finally, both $\cF$ and $\cI$ satisfy assumption \eqref{assump:nonzero} by \cref{lem:FIP1}.
Thus, all the assumptions stated before \cite[Theorem 2.14]{MMS_09} are satisfied.
To apply the theorem, notice that we have the following (stronger) properties:
\begin{itemize}
\item for all $(\cP',Z') \in \Stab_{\rep(G)}(\cD^G)$, 
\[
(\cI \circ \cF)(\cP'(\phi)) = \Bbbk[G] \otimes \cP'(\phi) \subseteq \cP'(\phi);
\]
\item for all $(\cP, Z) \in \Stab_{\vec_G}(\cD)$, 
\[
(\cF \circ \cI)(\cP(\phi)) = \bigoplus_{g \in G} g \otimes \cP(\phi) \subseteq \cP(\phi).
\]
\end{itemize}
As such, we conclude that 
\[
\Dom(\cI^{-1})\supseteq \Stab_{\rep(G)}(\cD^G); \quad \Dom(\cF^{-1}) \supseteq \Stab_{\vec_G}(\cD).
\]

Next, we show that $\cI^{-1}$ and $\cF^{-1}$ map to the correct codomain, i.e.
\[
\cI^{-1}(\Stab_{\rep(G)}(\cD^G)) \subseteq \Stab_{\vec_G}(\cD); \quad
\cF^{-1}(\Stab_{\vec_G}(\cD)) \subseteq \Stab_{\rep(G)}(\cD^G).
\]
Let $\sigma' = (\cP', Z') \in \Stab_{\rep(G)}(\cD^G)$. 
The fact that $\cI^{-1}\sigma$ is $G$-invariant follows from the fact that $\cI(X) \cong \cI(g \otimes X) \in \cD^G$ for all $g\in G$.

On the other hand, let $\sigma = (\cP, Z) \in \Stab_{\vec_G}(\cD)$.
Suppose $X \in \cF^{-1}\cP(\phi)$, so that $\cF(X) \in \cP(\phi)$.
Then for any $V\in \rep(G)$, we have $\cF(V \otimes X) = X^{\oplus \dim(V)}$.
Since $\cP(\phi)$ is an abelian category, it is closed under taking direct sums and hence $V \otimes X \in \cF^{-1}\cP(\phi)$.
Similarly, $Z(\cF(V\otimes X)) = \dim(V)\cdot Z(X)$, which shows that $\cF^{-1}\sigma$ is $\rep(G)$-equivariant.

Finally, we show that $\cI^{-1}$ and $\cF^{-1}$ are mutually inverse up to rescaling the central charge by $|G|\neq 0$.
The fact that $\cI^{-1}$ and $\cF^{-1}$ are biholomorphisms will then follow. 
Indeed, they are locally given by $\bbC$-linear morphisms on the spaces of central charges.
It is also immediate that rescaling central charges by $|G|$ is a $\bbC$-linear isomorphism.

Now suppose $\sigma = (\cP, Z) \in \Stab_{\vec_G}(\cD)$.
Using \cref{lem:FIproperties}(4) and the $G$-invariance of $Z$, we get $Z((\cF \circ \cI)(X)) = |G|\cdot Z$.
The $G$-invariance of $\sigma$ also guarantees that $X \in \cP(\phi) \iff g\otimes X \in \cP(\phi)$ for all $g \in G$.
Together with $\cP(\phi)$ being closed under taking direct sums and summands, we see that $(\cI^{-1} \circ \cF^{-1})\cP(\phi) = \cP(\phi)$.
As such, we get 
\[
(\cI^{-1}\circ \cF^{-1})\sigma = (\cP, |G|\cdot Z).
\]
On the other hand, suppose $\sigma' = (\cP', Z') \in \Stab_{\rep(G)}(\cD^G)$.
Using \cref{lem:FIproperties}(3) together with \cref{lem:Cactionpullback}(ii), we get
\[
(\cF^{-1} \circ \cI^{-1})\sigma' = (\cP', \FPdim(\Bbbk[G])\cdot Z') = (\cP', |G|\cdot Z').
\]
As such, $\cF^{-1}$ and $\cI^{-1}$ are mutually inverse up to rescaling the central charge by $|G|$.
\end{proof}

\section{Applications to \texorpdfstring{$\ell$}{l}-Kronecker quivers and free quotients} \label{sec:app}
As in the last section, we further assume throughout that $\characteristic(\Bbbk) \nmid |G|$.
\subsection{Stability conditions of \texorpdfstring{$\ell$}{l}-Kronecker quivers and McKay quivers}
Let $G$ be a finite group and let $V$ be a $G$-representation.
We first recall the definition of the associated separated McKay quiver.
\begin{definition}[\cite{AusRei_McKay}] \label{defn:McKayquiver}
The \emph{McKay quiver} $Q_{G,V}$ associated to a $G$-representation $V$ is the quiver with vertices given by the irreducible representations $\{S_i\}_{i=1}^m$ of $G$. Each vertex $S_i$ has $t_{i,j}$ arrows from $S_i$ to $S_j$, where $t_{i,j}$ is defined by $V \otimes S_i \cong \bigoplus_{j=1}^m S_j^{\oplus t_{i,j}}$.
The \emph{separated McKay quiver} $\overline{Q}_{G,V}$ is the separated quiver of the McKay quiver $Q_{G,V}$; this is the bipartite quiver with vertices $\{S_i\}_{i=1}^m \sqcup \{S_i'\}_{i=1}^m$ (double the vertices of $Q_{G,V}$). For each arrow $S_i \to S_j$ in $Q_{G,V}$, we have an arrow $S_i \to S_j'$ in $\overline{Q}_{G,V}$.
\end{definition}
Refer to \cref{fig:McKayquiver} for an example of a McKay quiver and its separated version. 

\begin{figure}
\text{McKay quiver: }
\begin{tikzcd}[column sep = small, row sep = small]
	\ydiagram{1,1,1} \\
	\ydiagram{2,1} \arrow[loop right]\\
	\ydiagram{3}
	\arrow[shift right, from=3-1, to=2-1]
	\arrow[shift right, from=2-1, to=1-1]
	\arrow[shift right, from=2-1, to=3-1]
	\arrow[shift right, from=1-1, to=2-1]
\end{tikzcd}
\qquad
\text{Separated McKay quiver: }
\begin{tikzcd}[column sep = small, row sep = small]
	\ydiagram{1,1,1} & \ydiagram{1,1,1} \\
	\ydiagram{2,1} & \ydiagram{2,1} \\
	\ydiagram{3} & \ydiagram{3}
	\arrow[from=2-1, to=2-2]
	\arrow[from=3-1, to=2-2]
	\arrow[from=1-1, to=2-2]
	\arrow[from=2-1, to=1-2]
	\arrow[from=2-1, to=3-2]
\end{tikzcd}
\caption{The McKay quiver vs.\ the separated McKay quiver of the symmetric group $S_3$ with respect to the standard two-dimensional representation \ydiagram{2,1}. The separated version can be viewed as the original one ``pulled apart''.}
\label{fig:McKayquiver}
\end{figure}

Let $\ell$ be the dimension of the $G$-representation $V$.
Now recall that a representation of the generalised $\ell$-Kronecker quiver $K_{\ell}$ with $\ell$ arrows
\[
K_\ell \coloneqq
\begin{tikzcd}
    1
    \arrow[r, draw=none, "\raisebox{+1.5ex}{\vdots}" description]
    \arrow[r, bend left,        "\alpha_1"]
    \arrow[r, bend right, swap, "\alpha_\ell"]
    &
    2
\end{tikzcd}
\]
is determined by $\ell$ morphisms $U_{\alpha_i} \in \Hom(U_1,U_2)$ for $1 \leq i \leq \ell$, (the sum of) which we shall equivalently view as an element $\sum_{i=1}^\ell e_i \otimes U_{\alpha_i} \in V \otimes \Hom(U_1, U_2)$ for some chosen basis $\{e_i\}_{i=1}^\ell$ of $V$.
As such, the $G$-representation $V$ determines an action of $G$ on $\rep(K_\ell)$.
Equivalently, the representation $V$ defines a $G$-action on the path algebra $\Bbbk K_\ell$ of $K_\ell$, hence an action on $\Bbbk K_\ell\mod \cong \rep(K_\ell)$.
We have the following result deduced from \cite{Dem_10,Dem_11}:
\begin{lemma}
The following categories are equivalent:
\begin{equation} \label{eqn:McKayandequivariantcat}
\rep(\overline{Q}_{G,V}) \cong  \rep(K_\ell)^G.
\end{equation}
\end{lemma}
\begin{proof}
By \cite[Proposition 2.48]{Dem_11}, we have the following equivalence:
\[
(\Bbbk K_\ell) G\lmod \cong (\Bbbk K_\ell \lmod)^G,
\]
where $(\Bbbk K_\ell) G$ denotes the skew group algebra (in the sense of \cite{RR_skew}; see also \cite{Dem_10}).
On the other hand, a direct computation shows that the quiver $Q_G$ (this is not the McKay quiver) in \cite[Theorem 1]{Dem_10} is the separated McKay quiver $\overline{Q}_{G,V}$ associated to $V$, hence we obtain the equivalence of categories
\[
\Bbbk \overline{Q}_{G,V}\lmod \cong (\Bbbk K_\ell) G\lmod.
\]
The rest follows from the standard identification between representations of quivers and modules over their path algebras.
\end{proof}
Note that by Auslander--Reiten's generalised McKay correspondence \cite[Theorem 1]{AusRei_McKay}, $\overline{Q}_{G,V}$ is a disjoint union of affine ADE quivers if and only if $\dim(V)=\ell=2$.

From now on, given a quiver $Q$, we shall denote $D^bQ \coloneqq D^b\rep(Q)$.
The generalised $\ell$-Kronecker quivers $K_\ell$ have well-studied stability manifolds:
\begin{itemize}
\item for $\ell= 1$ and $2$, we have $\Stab(D^bK_2) \cong \C^2$ by results in \cite{Qiu_PhDthesis,bridgeland_qiu_sutherland_2020} and \cite{Okada_P1} respectively; and 
\item for $\ell \geq 3$ we have $\Stab(D^bK_\ell) \cong \C \times \bbH$ , where $\bbH$ is the strict upper half plane \cite{DK_KroneckerStab}.
\end{itemize}
\begin{lemma} \label{lem:stabGwhole}
We have that $\Stab_{\vec_G}(D^bK_\ell)$ is the whole stability space $\Stab(D^bK_\ell)$.
\end{lemma}
\begin{proof}
We know that $\Stab_{\vec_G}(D^bK_\ell) = \Stab_G(D^bK_\ell) \subseteq \Stab(D^bK_\ell)$ is a closed subset from \cref{thm:equivariantisclosed}.
Moreover, $G$ acts on $K_0(D^bK_\ell)$ trivially since it does nothing to the simples of $\rep(K_\ell)$, so 
$
\Hom_\bbZ(K_0(D^bK_\ell), \C)^G = \Hom_\bbZ(K_0(D^bK_\ell), \C).
$
As such, $\Hom_\bbZ(K_0(D^b K_\ell),\C)^G$ is open, hence by \cref{thm:equivariantissubmfld}, 
$\Stab_{\vec_G}(D^bK_\ell)$ is an open subset of $\Stab(D^bK_\ell)$. 
Since $\Stab(D^bK_\ell)$ is connected, $\Stab_{\vec_G}(D^bK_\ell)$ must be the whole space.
\end{proof}
This leads to the following:
\begin{corollary}\label{cor:McKaystability}
Let $G$ be a finite group and $V \in \rep(G)$ with $\dim(V)=\ell$.
Together with the action of $G$ on $D^bK_\ell$ induced by $V$, we have that
\[
\Stab_{\rep(G)}(D^b\overline{Q}_{G,V}) \cong \Stab_{\vec(G)}(D^bK_\ell) = \Stab(D^bK_\ell) \cong 
	\begin{cases}
	\C^2, &\text{if } \ell \leq 2;\\
	\C \times \bbH, &\text{if } \ell \geq 3.
	\end{cases}
\]
\end{corollary}
\begin{proof}
We have $D^b\overline{Q}_{G,V} \cong D^bK_\ell^G$ from the identification in \eqref{eqn:McKayandequivariantcat}.
The rest follows from \cref{thm:moritastability} and \cref{lem:stabGwhole}, together with the identification of $\Stab(D^bK_\ell)$ given in \cite{Qiu_PhDthesis,bridgeland_qiu_sutherland_2020, Okada_P1, DK_KroneckerStab} for the cases $\ell = 1, \ell = 2$ and $\ell \geq 3$ respectively.
\end{proof}

\begin{example}
\begin{enumerate}
\item Following \cref{fig:McKayquiver}, we have shown that the submanifold of $\rep(S_3)$-equivariant stability conditions, $\Stab_{\rep(S_3)}(D^b\hat{D}_5)$, is homemorphic to $\C^2$, where $\hat{D}_5$ denotes an affine $D_5$ quiver.
By \cite[Theorem 1]{AusRei_McKay}, some other examples of (bipartite) affine quivers can be obtained from two-dimensional representations of finite groups $G$, such as finite subgroups $G \subset \SU(2)$ with the standard two-dimensional representation.
\item We present here an example for wild types.
Take $G = S_4$ and $V = \ydiagram{3,1}$ the standard 3 dimensional irreducible representation of $S_4$.
In this case, $\overline{Q}_{G,V}$ is the following wild quiver:
\[
\overline{Q}_{G,V} =
\begin{tikzcd}[column sep=13ex]
	\ydiagram{1,1,1,1} & \ydiagram{1,1,1,1} \\
	\ydiagram{2,1,1} & \ydiagram{2,1,1} \\
	\ydiagram{2,2} & \ydiagram{2,2} \\
	\ydiagram{3,1} & \ydiagram{3,1} \\
	\ydiagram{4} & \ydiagram{4}
	\arrow[from=5-1, to=4-2]
	\arrow[from=4-1, to=4-2]
	\arrow[from=4-1, to=3-2]
	\arrow[from=4-1, to=2-2]
	\arrow[from=4-1, to=1-2]
	\arrow[from=3-1, to=2-2]
	\arrow[from=3-1, to=4-2]
	\arrow[from=1-1, to=2-2]
	\arrow[from=2-1, to=2-2]
	\arrow[from=2-1, to=3-2]
	\arrow[from=2-1, to=4-2]
	\arrow[from=2-1, to=5-2]
\end{tikzcd}
=
\begin{tikzcd}
	& \bullet & \bullet \\
	& \bullet & \bullet \\
	\bullet &&& \bullet \\
	& \bullet & \bullet \\
	& \bullet & \bullet
	\arrow[from=2-3, to=2-2]
	\arrow[from=1-2, to=2-2]
	\arrow[from=2-3, to=1-3]
	\arrow[from=3-1, to=2-2]
	\arrow[from=3-1, to=4-2]
	\arrow[from=4-3, to=4-2]
	\arrow[from=2-3, to=4-2]
	\arrow[from=2-3, to=3-4]
	\arrow[from=4-3, to=2-2]
	\arrow[from=4-3, to=3-4]
	\arrow[from=5-2, to=4-2]
	\arrow[from=5-3, to=4-3]
\end{tikzcd}
\]
Using \cref{cor:McKaystability}, we get that $\Stab_{\rep(G)}(D^b\overline{Q}_{G,V})$ is biholomorphic to $\C \times \bbH$.
It would be interesting to identify the family of quivers that are obtainable as separated McKay quivers.
\end{enumerate}
\end{example}

\subsection{Stability conditions on free quotients}\label{subsection: free quotients}
Let $X$ be a smooth projective variety over $\C$. Throughout we consider \textit{numerical} stability conditions on $D^b\Coh(X)$, i.e. we assume that all central charges factor through the homomorphism to the numerical Grothendieck group, $K_0(X)\rightarrow\Knum(X)$. We also assume that all stability conditions satisfy support property -- from the discussion in \cref{section: support property} the results from earlier in the paper all hold in this context. Let $\Stab(X)\coloneqq\Stab(D^b\Coh(X))$.

\begin{definition}
	A Bridgeland stability condition $\sigma$ on $D^b\Coh(X)$ is called \textit{geometric} if for every point $x\in X$, the skyscraper sheaf $\cO_x$ is $\sigma$-stable.
\end{definition}

We denote by $\Stabgeo(X)$ the set of all geometric stability conditions. Note that by \cite[Proposition 2.9]{fuStabilityManifoldsVarieties2022}, when $\sigma\in\Stabgeo(X)$, all skyscraper sheaves have the same phase with respect to $\sigma$.

Now let $G$ be a finite group acting freely on $X$. This induces a right action of $G$ on $\Coh(X)$ by pullback. In this case, the $G$-equivariantisation $\Coh^G(X)\coloneqq \Coh(X)^G$ is the category of $G$-equivariant coherent sheaves.

Let $\pi\colon X\rightarrow Y\coloneqq X/G$ denote the quotient morphism. Since $G$ acts freely, $Y$ is smooth. Moreover, there is an equivalence
\begin{align*}
	\Psi\colon \Coh(Y) &\stackrel{\sim}{\longrightarrow} \Coh^G(X) \\
	E &\longmapsto (\pi^\ast(E),\lambdanat),
\end{align*}
where the isomorphisms $\lambdanat=\{\lambda_g\}_{g\in G}$ are given by the natural isomorphism
\begin{equation*}
	\lambda_g\colon \pi^\ast (E) \xrightarrow{\sim} g^\ast \pi^\ast (E) = (\pi\circ g)^\ast (E).
\end{equation*}

As before, $D^b\Coh^G(X)\cong (D^b\Coh(X))^G$. There are functors $\pi^\ast\colon D^b\Coh(Y) \rightarrow D^b\Coh(X)$ and $\pi_\ast\colon D^b\Coh(X)\rightarrow D^b\Coh(Y)$. Under $D^b\Coh(Y)\cong D^b\Coh^G(X)$, there is an isomorphism of functors $\cF\cong \pi^\ast$, $\cI\cong \pi_\ast$.

In this setting, \cref{thm:moritastability} preserves geometric stability conditions.
\begin{corollary}\label{cor:geometriccorrespondence}
	The functors $\pi^\ast$, $\pi_\ast$ induce biholomorphisms between the closed submanifolds of stability conditions $\Stab_{\rep(G)}(Y)$ and $\Stab_{\vec_G}(X)$,
	\[
		(\pi_\ast)^{-1}\colon \Stab_{\rep(G)}(Y) \overset{\cong}{\rightleftarrows} \Stab_{\vec_G}(X) \cocolon(\pi^\ast)^{-1},
	\]
	where they are mutually inverse up to rescaling the central charge by $|G|$. Moreover, this preserves geometric stability conditions.
	
	In particular, suppose $\sigma=(\cP_\sigma, Z_\sigma)\in\Stab_{\vec_G}(X)$. Then $(\pi^\ast)^{-1}(\sigma)\eqqcolon\sigma_{Y}=(\cP_{\sigma_{Y}},Z_{\sigma_{Y}})\in\Stab_{\rep(G)}(Y)$ is defined by:
	\begin{align*}
		\cP_{\sigma_{Y}}(\phi) &=\{E\in D^b\Coh(Y) : \pi^\ast(E)\in\cP_\sigma(\phi)\},\\
		Z_{\sigma_{Y}} &= Z_\sigma \circ \pi^\ast,
	\end{align*}
	where $\pi^\ast$ is the natural induced map on $\Knum(Y)$.
\end{corollary}

\begin{proof}
	The correspondence is a direct application of \cref{thm:moritastability}. Recall that being $\vec_G$-equivariant is equivalent to being $G$-invariant. When $G$ is abelian, being $\rep_G$-equivariant is equivalent to being $\widehat{G}$-invariant. In the abelian case, \cref{cor:geometriccorrespondence} is exactly \cite[Theorem 3.3]{dellStabilityConditionsFree2023}. The proof that geometric stability conditions are preserved in the abelian case does not use the $\widehat{G}$-action, so it generalises immediately.
\end{proof}

We now describe the $\rep(G)$-action in detail. Recall that $\pi_\ast\cO_X$ splits as a direct sum of vector bundles corresponding to the irreducible representations of $G$,
\begin{equation}\label{eqn: pushforward decomposition into irreps}
	\pi_\ast\cO_X = \bigoplus_{\rho\in\Irr(\rep(G))} E_{\rho}^{\bigoplus \dim \rho}.
\end{equation}
The action of $\rep(G)$ on $D^b\Coh(Y)$ is determined by $-\otimes E_{\rho}$. Alternatively, consider the classifying space $BG$ of $G$. Since $\rep(G)\cong\Coh(BG)$, the action of any $V\in\rep(G)$ on $\Coh(Y)$ can be equivalently described as taking the tensor product with the pullback of $V$ along the morphism $Y=X/G \rightarrow BG$.

Under the equivalence $\Psi\colon \Coh(Y)\rightarrow \Coh^G(X)$,
\begin{equation*}
	\Psi(\pi_\ast\cO_X)= (\pi^\ast \circ \pi_\ast \cO_X, \lambdanat) = (\cO_X, \id) \otimes \bbC[G].
\end{equation*}
The final equality can also be understood as an application of \cref{lem:FIproperties}(iii). The decomposition in \eqref{eqn: pushforward decomposition into irreps} corresponds to the decomposition of $\bbC[G]$ into irreducible representations,
\begin{equation*}
	(\cO_X, \id) \otimes \bbC[G] = \bigoplus_{\rho\in\Irr(\rep(G))} \left((\cO_X,\id)\otimes \rho \right)^{\oplus \dim \rho}.
\end{equation*}
In particular, $\Psi(E_\rho) = (\cO_X,\id)\otimes \rho$. Since pullbacks commute with tensor products and direct sums, the equivalence $\Psi$ is equivariant with respect to the action of $\rep(G)$, i.e. $\Psi(F\otimes E_\rho)\cong \Psi(F)\otimes \rho$ for all $F\in D^b\Coh(Y)$ and $\rho \in \Irr(\rep(G))$.

\begin{lemma}\label{lemma: Z always equivariant on free quotients}
	All central charges on $D^b\Coh(Y)$ are $\rep(G)$-equivariant, in particular
	\begin{equation*}
		\Hom_{K_0(\rep(G))}(K_0(Y),\C)=\Hom_\bbZ(K_0(Y),\C).		
	\end{equation*}
	In particular, $\Stab_{\rep(G)}(Y)$ is a union of connected components of $\Stab(Y)$.
\end{lemma}

\begin{proof}
	To prove that $\Hom_{K_0(\rep(G))}(K_0(Y),\C)=\Hom_\bbZ(K_0(Y),\C)$, it is enough to show that $\rho \cdot [F] = (\dim \rho) [F]$ for any $[F]\in K_0(X)$ and $\rho\in\Irr(\rep(G))$.

	Consider the Chern character map from the Grothendieck group to the Chow group, $\ch\colon K_0(X)\rightarrow  CH_*(X)$. By the Hirzebruch--Riemann--Roch theorem, this is injective. We claim that $\ch([E_\rho])=(\dim \rho, 0, \cdots, 0)$, where $[E_\rho]\in K_0(X)$. To prove this, first note that
	\begin{equation*}
		\Psi(\pi_\ast\cO_X\otimes E_\rho) = \left((\cO_X,\id)\otimes \bbC[G]\right)\otimes \rho = (\cO_X,\id)\otimes (\bbC[G]\otimes \rho).
	\end{equation*}
	Since $\bbC[G]\otimes \rho\cong \bbC[G]^{\oplus \dim \rho}$, we have
	\begin{equation*}
		\Psi(\pi_\ast\cO_X\otimes E_\rho) \cong \left((\cO_X,\id)\otimes \bbC[G] \right)^{\oplus \dim \rho} = \Psi(\left(\pi_\ast\cO_X\right)^{\oplus \dim \rho}).
	\end{equation*}
	It follows that $\pi_\ast\cO_X\otimes E_\rho\cong \left(\pi_\ast\cO_X\right)^{\oplus \dim \rho}$. Recall that $\ch([\pi_\ast\cO_X])=(|G|,0,\cdots, 0)$. Hence
	\begin{equation}\label{eqn: chern character expansion 1}
		\ch([\pi_\ast\cO_X\otimes E_\rho]) = \ch([\pi_\ast\cO_X]) \cdot \dim \rho = (|G|\dim\rho, 0, \cdots, 0).
	\end{equation}
	On the other hand,
	\begin{equation}\label{eqn: chern character expansion 2}
		\ch([\pi_\ast\cO_X\otimes E_\rho]) = \ch([\pi_\ast\cO_X])\cdot\ch([E_\rho]) = (|G|\ch_0(E_\rho), |G|\ch_1(E_\rho), \cdots, |G|\ch_n(E_\rho)).
	\end{equation}
	The claim follows from comparing \eqref{eqn: chern character expansion 1} and \eqref{eqn: chern character expansion 2}. Since $\ch$ is injective, it follows that, for any $[F]\in K_0(X)$, $\rho \cdot [F] = [F \otimes E_\rho]= (\dim \rho) [F]$ as required. 

	Since all central charges are $\rep(G)$-equivariant, by \cref{thm:equivariantissubmfld} it follows that $\Stab_{\rep(G)}(Y)$ is open in $\Stab(Y)$. Moreover, $\Stab_{\rep(G)}(Y)$ is also closed by \cref{thm:equivariantisclosed}, hence it is a union of connected components of $\Stab(Y)$.
\end{proof}

We can now generalise the arguments from the abelian case in \cite[Theorem 3.9, Corollary 3.10]{dellStabilityConditionsFree2023}.

\begin{theorem}\label{thm: Albanese connected component}
	Let $X$ be a smooth projective variety with finite Albanese morphism. Let $G$ be a finite group acting freely on $X$ and let $Y=X/G$. Then $\Stab^\dagger(Y):=\Stab_{\rep(G)}(Y)\cong \Stab_{\vec_G}(X)$ is a union of connected components consisting only of geometric stability conditions.
\end{theorem}

\begin{proof}
	$X$ has finite Albanese morphism, so it follows from \cite[Theorem 2.13]{fuStabilityManifoldsVarieties2022} that all stability conditions on $X$ are geometric. In particular, $\Stab_{\vec(G)}(X)\subset \Stabgeo(X)$. By \cref{cor:geometriccorrespondence}, $\Stab_{\vec_G}(X)\cong \Stab_{\rep(G)}(Y)\subset \Stabgeo(Y)$. The result now follows from \cref{lemma: Z always equivariant on free quotients}.
\end{proof}

\begin{example}[Calabi--Yau threefolds of abelian type]\label{example: CY 3fold of abelian type}
	A \textit{Calabi--Yau threefold of abelian type} is an \'etale quotient $Y=X/G$ of an abelian threefold $X$ by a finite group $G$ acting freely on $X$ such that the canonical line bundle of $Y$ is trivial and $H^1(Y,\C)=0$. These are classified in \cite[Theorem 0.1]{oguisoCalabiYauThreefolds2001}. In particular, $G$ is $(\Z/2\Z)^{\oplus 2}$ or $D_4$, and the Picard rank of $Y$ is 3 or 2 respectively.
	
	By the same discussion in \cite[Example 3.13]{dellStabilityConditionsFree2023}, Theorem \ref{thm: Albanese connected component} produces a non-empty union of connected components of $\Stab(Y)$, which is new in the case of $G=D_4$. 
\end{example}

\begin{example}[Generalised hyperelliptic varieties]\label{example: GHVs}
	A \textit{generalised hyperelliptic variety} $Y=A/G$ is a quotient of an abelian variety $A$ by the free action of a finite group $G$ which contains no translations. These varieties are K\"ahler and have Kodaira dimension zero. \cref{thm: Albanese connected component} applies to these varieties.
	
	In dimension 3 with $G=D_4$, $Y$ is a Calabi--Yau threefold of abelian type. Their construction in \cite{cataneseClassificationHyperellipticThreefolds2020} has been generalised to produce explicit examples of generalised hyperelliptic varieties in dimension $2n+1$ with $G=D_{4n}$ \cite[Theorem 2.1]{aguilovidalFreeDihedralActions2021}. In this case, $A=E^{2n}\times E'$, where $E,E'$ are elliptic curves. One can use the explicit description of the $D_{4n}$ action to prove that $A$ has non-finite Albanese morphism. 

	Another way to produce examples with $G$ non-abelian is to take semi-direct products of abelian groups acting freely on any abelian variety \cite[\S 7]{auffarthSmoothQuotientsComplex2022}.
\end{example}

In the surface case we can say more:
\begin{corollary}\label{cor: finite albanese surface quotient has connected component of geos}
	Let $X$ be a smooth projective surface with finite Albanese morphism. Let $G$ be a finite group acting freely on $X$. Let $S=X/G$. Then $\Stab^\dagger(S)=\Stabgeo(S)\cong \Stab_{\vec_G}(X)$. In particular, $\Stab^\dagger(S)$ is connected and contractible.
\end{corollary}

\begin{proof}
	Since $S$ is a surface, $\Stabgeo(S)$ is connected by \cite[Corollary 5.39]{dellStabilityConditionsFree2023} and contractible by \cite[Theorem A]{rekuskiContractibilityGeometricStability2023}. We conclude by \cref{thm: Albanese connected component}.
\end{proof}

\begin{example}[Non-abelian Beauville-type surfaces]\label{example: non abelian Beauville-type}
	Let $X=C_1\times C_2$, where $C_i\subset\mathbb{P}^2$ are smooth projective curves of genus $g(C_i)\geq 2$. Each curve has finite Albanese morphism, and hence so does $X$. Suppose there is a free action of a finite group $G$ on $X$, such that $S=X/G$ has $q(S):=h^1(S,\cO_S)=0$ and $p_g(S):=h^2(S,\cO_S)=0$. Then $\alb_S$ is trivial. This generalises a construction due to Beauville in \cite[Chapter X, Exercise 4]{beauvilleComplexAlgebraicSurfaces1996}, and we call $S$ a \textit{Beauville-type surface}. These are classified in \cite[Theorem 0.1]{bauerClassificationSurfacesIsogenous2008}. There are 17 families (depending on $G$).

By Corollary \ref{cor: finite albanese surface quotient has connected component of geos}, $\Stab(S)$ has a contractible connected component consisting only of geometric stability conditions. 
This provides a description of a connected component of $\Stab(S)$ for \emph{all} 17 families of Beauville-type surfaces; 5 of which that correspond to abelian groups $G$ were studied in \cite[Example 1.1, Corollary 3.10]{dellStabilityConditionsFree2023}, and our result now includes the 12 other families that correspond to non-abelian $G$. 
These are either $A_5$, $S_4$, $S_4\times \Z/2\Z$, $D_4\times \Z/2\Z$ (where $D_4$ is the dihedral group of order 8), or one of a list of certain higher order groups: $G(16)$, $G(32)$, $G(256, 1)$, $G(256, 2)$. 
Note that the final two are the only cases that have group elements which act by exchanging $C_1$ and $C_2$.
\end{example}

\appendix

\section{The support property}\label{section: support property}

Throughout, $\cD$ will be a triangulated category with a fixed group homomorphism to a lattice $v: K_0(\cD) \ra \Lambda$.
In this appendix we show that our main results (\cref{cor: C-equivariant complex submanifold} and \cref{thm:moritastability}) also hold in the setting of stability conditions with the support property.
We first recall the definition.
\begin{definition}\label{defn: support property}
	A (Bridgeland) stability condition $\sigma=(\cP,Z)$ on $\cD$ satisfies the \textit{support property} (with respect to $(\Lambda,v)$) if
	\begin{enumerate}
		\item $Z$ factors via a finite rank lattice $\Lambda$, i.e. $Z\colon K_0(\cD)\xrightarrow{v} \Lambda \rightarrow \C$, and
		\item there exists a quadratic form $Q$ on $\Lambda\otimes\bbR$ such that
		\begin{enumerate}
			\item $\Ker Z$ is negative definite with respect to $Q$, and
			\item every $\sigma$-semistable object $E\in \cD$ satisfies $Q(v(E))\geq 0$.
		\end{enumerate}
	\end{enumerate}
\end{definition}
\noindent Note that the support property implies locally-finiteness, see \cite[\S A]{bayerSpaceStabilityConditions2016} for details.

For the rest of this appendix, $\Stab(\cD)$ will instead denote the space of stability conditions satisfying the support property.
The result that the space of stability conditions forms a complex manifold (\cref{thm:stabmfld}) has the following refinement in this setting:
\begin{proposition}[{\cite[Proposition A.5]{bayerSpaceStabilityConditions2016}, \cite[Theorem 1.2]{bayerShortProofDeformation2019}}]\label{prop: support property deformation property}
	Let $\cD$ be a triangulated category. Assume $\sigma=(\cP,Z)\in\Stab(\cD)$ satisfies the support property with respect to $(\Lambda,v)$ and a quadratic form $Q$. Consider the open subset of $\Hom_\bbZ(\Lambda,\bbC)$ consisting of central charges whose kernel is negative definite with respect to $Q$, and let $U$ be the connected component containing $Z$. Let $\cZ$ denote the local homeomorphism from \cref{thm:stabmfld}, and let $\cU\subset\Stab(\cD)$ be the connected component of the preimage $\cZ^{-1}(U)$ containing $\sigma$. Then
	\begin{enumerate}
		\item the restriction $\cZ|_{\cU} \colon \cU\rightarrow U$ is a covering map, and 
		\item any stability condition $\sigma'\in\cU$ satisfies the support property with respect to the same quadratic form $Q$.
	\end{enumerate}
\end{proposition}

Now suppose $\cD$ is a triangulated module category over a fusion category $\cC$.
If we make the additional assumption in \cref{lem: openness of fusion-equivariant stability} that $\sigma$ satisfies the support property, then it has an alternate proof (avoiding the need of quasi-abelian categories) which we present below.
In what follows, a stability condition (satisfying the support property) defined via a stability function $Z$ on a heart $\cA$ will be denoted as $(Z, \cA)$.
\begin{lemma}
	Suppose $\sigma=(\cP,Z)\in\Stab(\cD)$ in \cref{prop: support property deformation property} is $\cC$-equivariant. Then there exists $V\subset U$ so that if $W\in V$ is $\cC$-equivariant, the lifted stability condition $\tau=(\cQ,W)\in\cU$, is also $\cC$-equivariant.
\end{lemma}

\begin{proof}
	Since $U$ is open, there exists $\varepsilon>0$ such that the open ball $V\coloneqq B_\varepsilon(Z)$ is contained in $U$. Therefore, there is a path $W_t$ in $V$ parametrised by $t\in [0,1]$ such that $W_0=Z$, $W_1=W$, and there exists $T\in[0,1]$ such that
	\begin{enumerate}
		\item $\Im W_t = \Im Z$ for $t\in[0,T]$,
		\item $\Re W_t = \Re W$ for $t\in[T,1]$.
	\end{enumerate}
	Furthermore, $\Hom_{K_0(\cC)}(\Lambda,\bbC)$ is a $\bbC$-linear subspace of $\Hom_\bbZ(\Lambda,\bbC)$, hence we can choose $W_t$ to be $\cC$-equivariant for all $t\in[0,1]$. By \cref{prop: support property deformation property}, $W_t$ lifts to a path $\tau_t=(\cQ_t,W_t)$ in $\cU$.
	
	Given a path in $\Stab(\cD)$ along which the imaginary part of the central charge is constant, the standard heart $\cP(0,1]$ also remains constant. For a proof of this see for example \cite[Corollary 5.2]{dellStabilityConditionsFree2023}. 
As such, by the construction of $W_t$ we have that 
\[
\cQ_t(0,1] = \cP(0,1] \qquad \text{for } t \in [0,T].
\]
For $t\in[T,1]$, since $\Re W_t$ is constant, we consider instead the ``rotated'' path of stability conditions $\frac{1}{2}\cdot \tau_t = (\frac{1}{2}\cdot\cQ_t, -iW_t)$. Along $\frac{1}{2}\cdot \tau_t$, for $t\in[T,1]$, the imaginary part of the central charge is constant:
\begin{equation*}
	\Im -iW_t = - \Re W_t = - \Re W_T = \Im -i W_T.
\end{equation*}
Hence the standard heart $(\frac{1}{2}\cdot \cQ_t)(0,1]=\cQ_t(1/2,3/2]$ is also constant, i.e.
\[
\cQ_t(1/2,3/2] = \cQ_T(1/2,3/2] \qquad \text{for } t\in[T,1].
\]

By \cref{thm: stab function over C and stab cond respecting C}, since $W_t$ is $\cC$-equivariant, $\tau_t = (W_t, \cP(0,1])$ is $\cC$-equivariant for all $t\in [0,T]$; in particular $\tau_T$ is $\cC$-equivariant and hence $\cQ_T(1/2, 3/2]$ is an abelian module category over $\cC$ (in fact, any $\cQ(a,a+1]$ is).
Using \cref{thm: stab function over C and stab cond respecting C} again we have $\frac{1}{2}\cdot\tau_t = (-iW_t, \cQ_T(1/2,3/2])$ is $\cC$-equivariant for all $t\in[T,1]$.
Since the $\bbC$-action on $\Stab(\cD)$ preserves $\cC$-equivariant stability conditions, $\tau_1 = (\cQ, W)$ is $\cC$-equivariant as required.
%
%
\end{proof}

In \cref{section: closed property}, we explained how to use an exact functor $\Phi\colon \cD \ra \cD'$ satisfying \eqref{assump:nonzero} to induce stability conditions from $\cD'$ to $\cD$. This also works with stability conditions with support property as follows. Given $\sigma' = (\cP', Z') \in \Stab(\cD')$ satisfying support property with respect to $(\Lambda', v')$, we define for each $\phi \in \R$ the following additive subcategory of $\cD$:
\[
\Phi^{-1}\cP'(\phi) \coloneqq \{ E \in \cD \mid \Phi(E) \in \cP(\phi) \},
\]
and group homomorphisms $\Phi^{-1}v'\coloneqq v' \circ\Phi\colon K_0(\cD)\rightarrow \Lambda'$ and $\Phi^{-1}Z'\coloneqq Z' \circ \Phi^{-1}v'$. As before, we have $\Phi^{-1}\sigma'\coloneqq (\Phi^{-1}\cP', \Phi^{-1}Z')\in\Stab(\cD)$, and this satisfies support property with respect to $(\Lambda', \Phi^{-1}v')$. 

As in \cite[Theorem 10.1]{bayerSpaceStabilityConditions2016}, our main results \cref{cor: C-equivariant complex submanifold} and \cref{thm:moritastability} go through with the support property using the above modification.

\printbibliography

\end{document}